\documentclass[11pt, leqno]{article}
\usepackage{amsmath,amssymb,amsbsy,amsfonts,amsthm,latexsym,
            amsopn,amstext,amsxtra,euscript,amscd,amsthm}

\newtheorem{lem}{Lemma}
\newtheorem{lemma}[lem]{Lemma}

\newtheorem{thm}{Theorem}
\newtheorem{theorem}[thm]{Theorem}

\def\\{\cr}
\def\({\left(}
\def\){\right)}
\def\[{\left[}
\def\]{\right]}
\def\<{\langle}
\def\>{\rangle}

\def\al{{\alpha}}
\def\be{{\beta}}
\def\ep{{\varepsilon}}

\def\cD{{\mathcal D}}

\def\cM{{\mathcal M}}

\begin{document}

\title{On members of Lucas sequences which are products of Catalan numbers}

\author{
{\sc Shanta~Laishram}\\
{Stat-Math Unit, Indian Statistical Institute}\\
{7, S. J. S. Sansanwal Marg, New Delhi, 110016, India}\\
{shanta@isid.ac.in}
\and
{\sc Florian~Luca}\\
{School of Mathematics, University of the Witwatersrand}\\
{Private Bag 3, Wits 2050, South Africa}\\
{Research Group in Algebraic Structures and Applications}\\
{King Abdulaziz University, Jeddah, Saudi Arabia}\\
{Centro de Ciencias Matem\'aticas UNAM, Morelia, Mexico}\\
{florian.luca@wits.ac.za}
\and
{\sc Mark~Sias}\\
{Department of Pure and Applied Mathematics}\\
{University of Johannesburg}\\
{PO Box 524, Auckland Park 2006, South Africa}\\
{msias@uj.ac.za}}

\date{} %\today}
\pagenumbering{arabic}

\maketitle

\begin{abstract}
We show that if $\{U_n\}_{n\ge 0}$ is a Lucas sequence, then the largest $n$ such that 
$|U_n|=C_{m_1}C_{m_2}\cdots C_{m_k}$ with $1\le m_1\le m_2\le \cdots\le m_k$, where $C_m$ is the $m$th Catalan number 
satisfies $n<6500$. In case the roots of the Lucas sequence are real, we have $n\in \{1,2, 3, 4, 6, 8, 12\}$.  As a consequence, 
we show that if $\{X_n\}_{n\geq 1}$ is the sequence of the $X$ coordinates of a Pell equation $X^2-dY^2=\pm 1$ with a 
nonsquare integer $d>1$,  then $X_n=C_m$ implies $n=1$.  
\end{abstract}

\section{Introduction}

Let $r,~s$ be coprime nonzero integers with $r^2+4s\ne 0$. Let $\alpha,~\beta$ be the roots of the quadratic equation $\lambda^2-r\lambda-s=0$ 
and assume without loss of generality that $|\al|\geq \be|$. We assume further that $\alpha/\beta$ is not a root of $1$. The Lucas 
sequences $\{U_n\}_{n\ge 0}$ and $\{V_n\}_{n\ge 0}$ of parameters $(r,s)$ are given by 
$$
U_n=\frac{\alpha^n-\beta^n}{\alpha-\beta}\qquad {\text{\rm and}}\qquad V_n=\alpha^n+\beta^n\qquad {\text{\rm for~all}}\qquad n\ge 0.
$$
Alternatively, they can be defined recursively as 
$$
U_{n+2}=rU_{n+1}+sU_n\quad {\text{\rm and}}\quad V_{n+2}=rV_{n+1}+sV_n\qquad  {\text{\rm for all}}\quad n\ge 0
$$ 
with initial conditions $U_0=0,~U_1=1,~V_0=2,~V_1=r$. In case when $r=s=1$, we get $U_n=F_n$, the $n$th Fibonacci number. 
Let 
$$
B_m:=\binom{2m}{m}\quad {\text{\rm and}}\quad C_m:=\frac{1}{m+1}\binom{2m}{m}\quad {\text{\rm for}}\quad m\ge 0,
$$
be the middle binomial coefficient and Catalan number, respectively. For each $m$, we write $D_m$ for one of the numbers $B_m,C_m$. 
Let 
$$
{\mathcal PBC}:=\{\pm \prod_{j=1}^k D_{m_j}: D_m\in \{B_m,C_m\}, ~k\ge 1,~ 1\le m_1\le m_2\le\cdots \le m_k\}
$$
be the set of integers which are products of middle binomial coefficients and Catalan numbers. 
Diophantine equations with members of ${\mathcal PBC}$ have been studied before. For example, in \cite{LuOd}, the authors characterised all nontrivial solutions of the system of two equations 
$$
\sum_{i=1}^n ip_i=\sum_{j=1}^r jq_j\quad {\text{\rm and}}\quad \prod_{i=1}^n B_{i}^{p_i}=\prod_{j=1}^r B_j^{q_j}.
$$   
This system of equations arose naturally from a question in topology concerning $n$-dimensional complexes which do not embed in ${\mathbb R}^{2n}$ and characterising non-homotopic pairs of such with 
the same homology. 
In \cite{Lu}, it was shown that the largest positive integer solution $(n,m)$ of the Diophantine equation
$$
F_n=C_m
$$
is $(n,m)=(5,3)$. In \cite{LuPa}, it is shown that if $\{u_n\}_{n\ge 0}$ is any nondegenerate binary recurrence of integers, then the equation $u_n=B_m$ 
has only finitely many positive integer solutions $(n,m)$. Inspired by these problems, we study here the Diophantine equation obtained by imposing that a member of the Lucas sequences $U_n$ or $V_n$ is a product of 
middle binomial coefficients of Catalan numbers. 

Our theorem is the following.

\begin{theorem}\label{theorem1}
For each $m$, let $D_m\in \{B_m,C_m\}$. The equation  
\begin{align}\label{eqn1}
U_n=\pm D_{m_1} D_{m_2}\cdots D_{m_k},\quad {\text{where}}\quad k\ge 1 \quad {\text{and}}\quad 1\le m_1\le \cdots\le m_k,
\end{align}
implies $n<6500$ if $n$ is odd and $n\leq 720$ if $n$ is even. Further when $\alpha, \beta $ are real, then 
$n\in \{1,2, 3, 4, 6, 8, 12\}$.  

The equation
\begin{align}\label{eqn2}
V_n=\pm D_{m_1} D_{m_2}\cdots D_{m_k},\quad {\text{where}}\quad k\ge 1 \quad {\text{and}}\quad 1\le m_1\le \cdots\le m_k,
\end{align}
implies $n<6500$ and $4\nmid n$. Further, when $\alpha, \beta $ are real, then $n\in \{1, 2, 3, 6\}$.    
\end{theorem}

Note that $U_1=1\in {\mathcal PBC}$. For this reason, whenever we look at equation \eqref{eqn1}, we omit $n=1$ and assume $n\ge 2$. 

We present a corollary regarding $X$-coordinates of Pell equations which are in $\{C_m, D_m\}$. 
For a positive integer $d$ which is square-free, let $(X_n,Y_n)$ be the $n$-th solution of the Pell equation 
$X^2-dY^2=\pm 1$ in positive integers $(X,Y)$(solution of either  $X^2-dY^2=1$ or $X^2-dY^2=-1$, not separately).  
Arithmetic properties of the coordinates $X$ or $Y$ of Pell equations have been studied before. For example, values 
of $n$ such that $X_n$ is a square have been studied by Ljunggren \cite{Lj}. He proved that there are at most two 
such values of $n$. This was improved later in \cite{Schi} where it was shown that in fact there is at most one such $n$ 
except for $d=1785$, for which both $X_1$ and $X_2$ are squares. in \cite{LaLuSi}, a similar result was proved for 
$X_n$ being a product of factorials. We supplement this with the following result on values of $X_n$ which are 
in $\{C_m, B_m\}$. 

\begin{theorem}\label{thm:3}
Let $(X_n,Y_n)$ be the $n$th solution in positive integers of the equation $X^2-dY^2=\pm 1$ for some squarefree integer 
$d$. Then $X_n\in \{C_m, B_m\}$ implies $n=1$. Similarly, let $(W_n, Z_n)$ be the $n$th solution in positive integers of the 
equation $W^2-dZ^2=\pm 4$ for some squarefree integer $d$. Then $W_n\in \{C_m, B_m\}$ implies $n\in \{1, 3\}$ or 
$n=2$ with 
\begin{align*}
&d=2, W_2=B_2=6: \ \  \ 6^2-2\cdot 4^2=4, \ {\rm where} \ (W_1,Z_1)=(2, 2); \\
&d=2, W_2=C_4=14: \ 14^2-2\cdot 10^2=-4, \ {\rm where} \ (W_1,Z_1)=(2, 2);\\
&d=3, W_2=C_4=14: \ 14^2-3\cdot 8^2=4, \ {\rm where} \ (W_1,Z_1)=(4, 2). 
\end{align*}
\end{theorem}

We believe that there are only finitely many solutions of \eqref{eqn1} such that $n\in \{6, 8, 12\}$ regardless of whether 
$\alpha,\beta$ are real or complex conjugates, which we are not able to prove.  Also we conjecture that 
there are only finitely many solutions of \eqref{eqn2} with $n\in \{3, 6\}$. Recently, the three of us proved similar theorems for 
members of Lucas sequences $U_n,~V_n$ which are products of factorials in \cite{LaLuSi}. The current paper is much 
inspired by the method of the paper \cite{LaLuSi}. 

We give the proof of Theorem \ref{theorem1} in Section 4 and the proof of Theorem \ref{thm:3} in Section 5.  Throughout the 
paper, we use $P(n), \mu(n)$ and $\varphi(n)$ with the regular meaning as being the largest prime factor of $n$, the M\"obius function of $n$ and the Euler phi function 
of $n$, respectively.   
All the computations in this manuscript were carried out in SageMath. 

\section{Preliminaries}

Let $n_0$ be a positive integer. For an integer $\ell$, define
\begin{equation}\label{Mn}
M_{n_0}(\ell):=\log \left( \prod_{\substack{p^{\nu_p}\|\ell \\ 
p\equiv \pm 1\pmod {n_0}}} p^{\nu_p}\right)=\sum_{\substack{p^{\nu_p}\|\ell \\ 
p\equiv \pm 1\pmod {n_0}}} \nu_p\log p.
\end{equation}
We prove a number of results to estimate lower and upper bounds for $M_{n_0}(U_n)$ and 
$M_{n_0}(V_n)$ for some divisors $n_0$ of $n$.  

To recall the terminology, we take coprime nonzero integers $r, s$ with $r^2+4s\neq 0$ and  let 
$\alpha$ and $\beta$ be the roots of the equation $\lambda^2-r\lambda-s=0$. For $n\geq 0$, we have 
\begin{align*}
U_n=\frac{\alpha^n-\beta^n}{\alpha-\beta} \quad {\rm and} \quad 
V_n= \alpha^n+\beta^n.
\end{align*}
We suppose that $\alpha/\beta$ is not a root of unity.  We 
assume without loss of generality that $|\alpha|\ge |\beta|$. Further, we may 
replace $(\alpha,\beta)$ by $(-\alpha,-\beta)$ if needed. This replacement changes the pair 
$(r,s)$ to $(-r,s)$, while $|U_n|$ and $|V_n|$ are not affected 
and hence the values of $M_{n_0}(|U_n|)$ and $M_{n_0}(|V_n|)$ for any divisor $n_0$ of $n$. 
Thus, we may assume that $r>0$. When $\alpha,\beta$ are real, these conventions imply that $\alpha$ is positive 
so $\alpha>|\beta|$.  Further, in this case $U_n>0$ and $V_n>0$ for all $n\ge 1$. 

We begin by proving a lower bound for $M_{n_0}(U_n)$ and $M_{n_0}(V_n)$ for some divisors $n_0$ of $n$. Throughout the paper, we use $x:=\beta/\alpha$. 

\begin{lemma}\label{primn_0}
Let $n$ be a positive integer and $p<p_1$ be distinct primes and $t\geq 0, h>0, h_1>0$ be integers. 
Let $n_0\in \{p^h, p^hp^{h_1}_1\}$, $n_0>4, n_0\notin \{6, 12\}$ be such that $n_0p^t\mid n$.  
 Then 
\begin{align}\label{n_0U}
\begin{split}
 &M_{n_0}(U_n)\geq \\
&\begin{cases}
n\left(1-\frac{1}{p^{t+1}}\right)\log |\al|+ \log\left( \frac{1-x^n}{1-x^{n/p^{t+1}}}\right)-\log(p^{t+1}), &  n_0=p^{h};\\
n\left(1-\frac{1}{p^{t+1}}\right)\left(1-\frac{1}{p_1}\right)\log |\al| +\log\left( \frac{1-x^n}{1-x^{\frac{n}{p^{t+1}}}}\right)-\log (pp_1)^{t+1},
&  n_0=p^hp^{h_1}_1,
\end{cases} 
\end{split}
\end{align}
and for $n_0=p^h, p>2$, 
\begin{align}\label{n_0V}
\begin{split}
M_{n_0}(V_n) \geq 
n\left(1-\frac{1}{p^{t+1}}\right)\log |\al|+ \log\left( \frac{1+x^n}{1+x^{n/p^{t+1}}}\right)-\log(p^{t+1}).
\end{split}
\end{align}
\end{lemma}

\begin{proof}
Let $n_0$ be the divisor of $n$ given in the statement of the lemma. Let  $m=n_0p^t$. Write
$$U_n=\frac{\al^n-\be^n}{\al-\be}=\left(\frac{(\al^{n/m})^{m}-(\be^{n/m})^{m}}
{\al^{n/m}-\be^{n/m}}\right)\left( \frac{\al^{n/m}-\be^{n/m}}{\al-\be}\right).
$$
Let  $\al_1:=\al^{n/m}$ and $\be_1:=\be^{n/m}$ and put 
$$
U^{1}_{\ell}=\frac{\al^{\ell}_1-\be^{\ell}_1}{\al_1-\be_1}  \quad {\rm and} \quad 
V^{1}_{\ell}=\al^{\ell}_1+\be^{\ell}_1 \qquad {\rm for} \quad \ell\geq 1.
$$ 
Then $\{U^{1}_{\ell}\}_{\ell\geq 0}$ and$\{V^{1}_{\ell}\}_{\ell\geq 0}$  are the Lucas sequences 
with parameters $(r_1, s_1)$, where $(r_1, s_1)=(\al_1+\be_1, -\al_1\be_1)=(V_{n/m}, (-1)^{n/m-1}s^{n/m})$.  Further, we have 
 $U_n=U^1_{m}U_{n/m}$ and $V_n=V^1_{m}$ implying 
 $$M_{n_0}(U_n)\geq M_{n_0}(U^1_{m}) \qquad {\rm and} \qquad  M_{n_0}(V_n)\geq M_{n_0}(V^1_{m}).$$ 
Observe that $U^1_{m}=U^1_{n_0p^t}$ is divisible by each 
$U^1_{n_0p^i}, 0\leq i\leq t$. Recall that a prime $q\mid U^1_{\ell}$ is a 
primitive divisor of $U^1_{\ell}$ if $q\nmid U^1_{\ell'}$ for $\ell'<\ell$ and $q\nmid r^2_1+4s_1$. 
Also the primitive divisors of $U^1_{\ell}$ are all congruent to one of $ \pm 1$ modulo $\ell$. Hence, 
the primitive divisors of  $U^1_{n_0p^i}$ for $0\leq i\leq t$ are all congruent to one of 
$\pm 1$ modulo $n_0$. We now look at the primitive part of $U^1_{\ell}$. This is the part of 
$U^1_{\ell}$  built up only with powers of primitive prime divisors of $U^{1}_{\ell}$.  Thus, 
the primitive parts of $U^1_{n_0p^i}$ for $0\leq i\leq t$ divide $U^1_{m}$. Hence, 
$$M_{n_0}(U_n)\geq M_{n_0}(U^1_{m})\geq M_{n_0}\left(\prod^t_{i=0}U^1_{n_0p^i}\right).$$

For a positive integer $\ell$, let 
\begin{align*}
\Phi_\ell(\alpha_1, \beta_1):=\prod_{\substack{1\le k\le \ell \\ (k, \ell)=1}} (\alpha_1-e^{2\pi i k/\ell} \beta_1)
\end{align*}
be the specialisation of the homogenization $\Phi_\ell(X,Y)$ of the $\ell$-th cyclotomic polynomial $\Phi_\ell(X)$ in 
the pair $(\alpha_1, \beta_1)$. Further, it is well-known (see, for example,  \cite[Theorem 2.4]{BHV}), that 
for $\ell>4, \ell\notin \{6, 12\}$, 
$$
\prod_{\substack{p^{\nu_p}\| U^1_{\ell}\\ p~{\rm primitive}}} p^{\nu_p}=\frac{\Phi_{\ell}(\alpha_1,\beta_1)}{\delta_\ell},
$$
where $\delta_\ell\in \{1,2, P(\ell)\}$. Since primitive divisors of 
$U^1_{\ell}$ are congruent to one of $\pm 1$ modulo $\ell$, we obtain by taking $\ell=n_0p^i$ for $0\leq i\leq t$ that 
\begin{align}\label{MUn}
\begin{split}
M_{n_0}(U_n)\geq M_{n_0}\left(\prod^t_{i=0}U^1_{n_0p^i}\right)\geq 
\left(\prod^t_{i=0}|\Phi_{n_0p^i}(\alpha_1, \beta_1)|\right)(P(n_0))^{-t-1}.
\end{split}
\end{align}
Also from the fact that $V_n=V^1_{n_0p^t}$ is divisible by each $V_{n_0p^i}, 0\leq i\leq t$ (here $n_0, p$ are both odd) 
and the primitive part of $V_{n_0p^i}$ is exactly the primitive part of $U^1_{2n_0p^i}$, we obtain  similarly 
\begin{align}\label{MVn}
M_{n_0}(V_n)\geq 
M_{2n_0}\left(\prod^t_{i=0}U^1_{2n_0p^i}\right)\geq \left(\prod^t_{i=0}|\Phi_{2n_0p^i}(\alpha_1, \beta_1)|\right)
(P(n_0))^{-t-1}.
\end{align}
Therefore, it remains to estimate the right--hand sides of inequalities \eqref{MUn} and \eqref{MVn}.  

It is well-known that for a positive integer $\ell$, 
\begin{align*}
\Phi_\ell(\alpha_1, \beta_1) = \prod_{d\mid \ell} (\alpha^{\frac{\ell}{d}}_1-\beta^{\frac{\ell}{d}}_1)^{\mu(d)}.
\end{align*}
Hence, we have, by using $\al^{n_0p^t}_1=\al^n$,
\begin{align}\label{MU1}
\begin{split}
\prod^t_{i=0}\Phi_{n_0p^i}(\alpha_1, \beta_1)=&\prod^t_{i=0}\frac{\al^{p^{h+i}}_1-\be^{p^{h+i}}_1}
{\al^{p^{h+i-1}}_1-\be^{p^{h+i-1}}_1}
=\frac{\al^{p^{h+t}}_1-\be^{p^{h+t}}_1}{\al^{p^{h-1}}_1-\be^{p^{h-1}}_1}\\
=&\frac{\al^n-\be^n}{\al^{n/p^{t+1}}-\be^{n/p^{t+1}}}, \quad n_0=p^h;
\end{split}
\end{align}
and 
\begin{align}\label{MU2}
\begin{split}
\prod^t_{i=0}\Phi_{n_0p^i}(\alpha_1,\beta_1)=&\prod^t_{i=0}\frac{(\al^{p^{h+i}p^{h_1}_1}_1-\be^{p^{h+i}p^{h_1}_1}_1)
(\al^{p^{h+i-1}p^{h_1-1}_1}_1-\be^{p^{h+i-1}p^{h_1-1}_1}_1)}
{(\al^{p^{h+i-1}p^{h_1}_1}_1-\be^{p^{h+i-1}p^{h_1}_1}_1)(\al^{p^{h+i}p^{h_1-1}_1}_1-\be^{p^{h+i}p^{h_1-1}_1}_1)}\\
=&\frac{(\al^{p^{h+t}p^{h_1}_1}_1-\be^{p^{h+t}p^{h_1}_1}_1)(\al^{p^{h-1}p^{h_1-1}_1}_1-\be^{p^{h-1}p^{h_1-1}_1}_1)}{
(\al^{p^{h-1}p^{h_1}_1}_1-\be^{p^{h-1}p^{h_1}_1}_1)(\al^{p^{h+t}p^{h_1-1}_1}_1-\be^{p^{h+t}p^{h_1-1}_1}_1)} \\
=&\left(\frac{\al^n-\be^n}{\al^{\frac{n}{p^{t+1}}}-\be^{\frac{n}{p^{t+1}}}}\right)\left(
\frac{\al^{\frac{n}{p_1p^{t+1}}}-\be^{\frac{n}{p_1p^{t+1}}}}{\al^{n/p_1}-\be^{n/p_1}}\right),
 \quad  n_0=p^hp^{h_1}_1. 
\end{split}
\end{align}
Also,
\begin{align}\label{MV1}
\prod^t_{i=0}\Phi_{2n_0p^i}(\alpha_1,\beta_1)
=\frac{\al^{p^{h+t}}_1+\be^{p^{h+t}}_1}{\al^{p^{h-1}}_1+\be^{p^{h-1}}_1}
=\frac{\al^n+\be^n}{\al^{\frac{n}{p^{t+1}}}+\be^{\frac{n}{p^{t+1}}}},\quad n_0=p^h.
\end{align}

From $|\al|\geq |\be|$, we have $|x|\leq 1$. Taking out the powers of $\al$ in 
\eqref{MU1}--\eqref{MV1} and further using in \eqref{MU2} the inequality
\begin{align*}
\left|\frac{1-y}{1-y^{p^{t+1}}}\right|\geq \frac{1}{p^{t+1}} \quad {\text{\rm valid~for~all}}\quad p, \quad {\text{\rm {\rm where}}} \quad 
y:=x^{\frac{n}{p_1p^{t+1}}}\quad {\text{\rm has}}\quad |y|\leq 1, 
\end{align*}
we get the assertions \eqref{n_0U} and \eqref{n_0V} from \eqref{MUn} and \eqref{MVn}, respectively.  
\end{proof}

From  the inequality
$$44.72(\log t+2.36)^2+0.16\log ^2t\leq 44.88\log^2t+211.08\log t+249.08,$$
we obtain the following result which is \cite[Lemma 4]{LaLuSi} and which is a consequence of Voutier \cite[Lemma 5]{Vou}.  

\begin{lemma}\label{-44.72}
Let $\al$ and $\be$ be complex conjugates with $\log |\al|>4$. Let 
\begin{align}\label{f(x)}
f(\ell):=44.88\log^2 \ell+211.08\log \ell+249.08 \quad {\rm for} \quad \ell>1.
\end{align}
Then for integer $\ell\ge 3$, we have 
\begin{equation}
\label{a-b}
\log |\alpha^\ell-\beta^\ell| \ge  \log |\alpha|\left(\ell-f\left(\frac{\ell}{\gcd(\ell, 2)}\right)\right)
\end{equation}
and 
\begin{align}\label{a+b}
\log |\alpha^\ell+\beta^\ell|\ge \log |\alpha|\left(\ell-f(\ell)\right).
\end{align}
\end{lemma}

The following lemma gives us range for the parameters $(r, s)$ in case when $\al$ is real, positive and lies in an interval $[c_1,c_2]$. 

\begin{lemma}\label{rsBD}
Let $\al, ~\be$  be real. Assume $\alpha>0$. Let $c_1\leq \al\leq c_2$ where $c_1, c_2$ are positive reals and $r^2+4s>0$. 
For $s>0$, we have 
$r<c_2$ and 
\begin{align*}
\max\left\{c_1(c_1-r), \frac{c^2_1-r^2}{4}\right\}\leq s\leq c_2(r-c_2). 
\end{align*}
For $s<0$, we have $c_1\leq r\leq 2c_2$ and 
\begin{align*}
c_2(r-c_2)\leq |s|<\frac{r^2}{4}, \quad {and \ further} \quad 
|s|<c_1(r-c_1) \quad {if}  \quad r<2c_1.
\end{align*}
\end{lemma}

\begin{proof}
We have $2c_1\leq 2\al=r+\sqrt{r^2+4s}\leq 2c_2$. This gives the inequality
$r^2+4s\leq (2c_2-r)^2$ implying $s\leq c_2(c_2-r)$. If $2c_1>r$, we then have $r^2+4s\geq (2c_1-r)^2$ 
giving $s\geq c_1(c_1-r)$. 

Let $s>0$. Then $r<\al\leq c_2$ giving $r<c_2$ and $s\leq c_2(c_2-r)$. If $c_1>r$, then $2c_1>r$ and therefore 
 $s\geq c_1(c_1-r)$. Also 
 $$
 2c_1\leq r+\sqrt{r^2+4s}\leq 2\sqrt{r^2+4s}
 $$ 
 gives 
${\displaystyle{s\geq \frac{c^2_1-r^2}{4}}}$ implying 
$$
s\geq \max\left\{c_1(c_1-r), \frac{c^2_1-r^2}{4}\right\}.
$$
Let $s<0$. Then $c_1\leq \al<r<r+\sqrt{r^2+4s}\leq 2c_2$ giving $c_1<r<2c_2$. Also $r^2+4s>0$ gives 
$|s|=-s<{r^2}/{4}$. From $s\leq c_2(c_2-r)$, we get 
$$
|s|=-s\geq c_2(r-c_2).
$$ 
If $r<2c_1$, then 
$s\geq c_1(c_1-r)$ implying $|s|=-s\leq c_1(r-c_1)$. 
\end{proof}

The following lemma is proved using Stirling's formula. 

\begin{lemma}\label{m14} The function
$m\mapsto {\log(C_m/2)}/{m}$ is increasing  for $m\geq 7$. Hence, 
\begin{align}\label{m1.3}
\log \left(\frac{B_m}{2}\right)>\log \left(\frac{C_m}{2}\right)>\begin{cases}
m \quad &{\rm for} \quad  m\geq 14;\\
1.36m \quad &{\rm for} \quad  m\geq 400;\\
1.38m &{\rm for} \quad  m\geq 2100.
\end{cases}
\end{align}
Further, given $M\geq 7$ and $m\leq M$, we have 
\begin{align}\label{0001}
\frac{m\log 2m}{\log (C_m/2)}\frac{\log (C_M/2)}{M}\leq 1.0001 \log 2M.
\end{align}

\end{lemma}

\begin{proof}
We recall Stirling's formula. For a positive integer $\nu $, we have 
\begin{align*}
\sqrt{2\pi \nu}~e^{-\nu}\nu ^{\nu }e^{\frac{1}{12\nu+1}} <\nu !< 
\sqrt{2\pi \nu}~e^{-\nu} \nu ^{\nu } e^{\frac{1}{12\nu}}.
\end{align*}
From $C_m=\frac{(2m)!}{(m+1)(m!)^2}$,  we have 
\begin{align}\label{<cm<}
m\log 4 -\sigma_m<\log (C_m/2)<m\log 4 -\tau_m,
\end{align}
where 
\begin{align*}
&\sigma_m:=\log 2+\log (m+1)+\log \sqrt{\pi m}+\frac{1}{6m}-\frac{1}{24m+1} \\
{\rm and} \quad & \tau_m:=\log 2+\log (m+1)+\log \sqrt{\pi m}+\frac{2}{12m+1} -\frac{1}{24m}.
\end{align*}
We have $C_m<{4^m}/{\sqrt{\pi m}}$ and 
\begin{align*}
\frac{4^m}{\sqrt{\pi m}}\left(\frac{m+2}{4m+2}\right)^m=\frac{\left(1+{3}/{(2m+1)}\right)^m}{\sqrt{\pi m}}
\leq \frac{e^{\frac{3m}{2m+1}}}{\sqrt{\pi m}}<\frac{e^{3/2}}{\sqrt{\pi m}}<1 \quad {\rm for} \quad m\geq 7.
\end{align*}
Hence, from $C_{m+1}/C_m={(4m+2)}/{(m+2)}$, we get 
\begin{align*}
m\log \left(\frac{C_{m+1}}{2}\right)-(m+1)\log \left(\frac{C_{m}}{2}\right)\geq m\log \left(\frac{C_{m+1}}{C_m}\right)-\log C_m>0 
\end{align*}
for $m\ge 7$. This shows that ${\log(C_m/2)}/{m}$ is an increasing function for $m\geq 7$.  Hence, the assertion \eqref{m1.3}  follows by 
calculating ${\log(C_m/2)}/{m}$ at $m=14, 400, 2100$, respectively.   

From \eqref{<cm<}, we have 
\begin{align*}
\frac{m\log 2m}{\log ({C_m}/{2})}\leq \frac{\log 2m}{\log 4-{\sigma_m}/{m}}
\end{align*}
and the right--hand side is an increasing function of $m$. Therefore, 
from $m\leq M$ and inequality \eqref{<cm<} again, we get 
\begin{align*}
\left(\frac{m\log 2m}{\log (\frac{C_m}{2})}\right)\left(\frac{\log ({C_M}/{2})}{M}\right)&\leq \frac{\log 2M}{\log 4-{\sigma_M}/{M}}(\log 4-{\tau_M}/{M})\\
&=(\log 2M)\left(1+\frac{\sigma_M-\tau_M}{M\log 4-\sigma_M}\right)\\
&\leq (\log 2M)\left(1+\frac{\frac{1}{24M+1}+\frac{4}{12M+1}}{24M(M\log 4-\sigma_M)}\right)\\
&\leq 1.0001\log 2M,
\end{align*}
since $M\geq 7$, implying the assertion \eqref{0001}. 
\end{proof}

The next lemma follows easily from the Brun-Titchmarsh inequality given by Montgomery and 
Vaughan \cite[Theorem 2]{MV} since $\pi(1;q,l)=0$ and $\pi(y; q,l)\leq \pi(y+1; q,l)-\pi(1; q, l)$. 
Recall that $\pi(y; q, l)$ stands for the number of primes $p\leq y$ and $p\equiv l\pmod q$. 

\begin{lemma}\label{pix}
Let $q$ be a positive integer,  $l$ be coprime to $q$ and $y>q$. Then
\begin{align*}
\pi(y;q,l)\le \frac{2y}{\varphi(q)\log (y/q)} \quad {\rm and} \quad \pi(2y;q,l)-\pi(y;q,l)\le \frac{2y}{\varphi(q)\log (y/q)}.
\end{align*}
\end{lemma}

As usual, let 
$$\psi(y; q, l):=\sum _{\substack{p^t\leq y \\ p\equiv l\pmod {m}}}\log p \quad {\rm and} \quad 
\theta(y; q, l):=\sum _{\substack{p\leq y \\ p\equiv l\pmod {m}}}\log p.$$
The following estimates are from \cite[Table 2]{RaRu}. We have taken into account the estimates 
for $\theta\#$ defined in \cite[Table 2]{RaRu} for $q\in \{8, 16, 24\}$. 

\begin{lemma}\label{pi<24}
Let $q\in \{8, 9, 12, 16, 24\}$ or $5\leq q\leq 23$ be a prime and  $\ell_0$ be an integer coprime to $q$ with 
$\ell_0 \not\equiv 1\pmod q$. Then for $y\geq q$, we have 
\begin{align}
\begin{split}
&\psi(y; q, 1)+\psi(y; q, \ell_0)\leq  \frac{2y}{\varphi(q)}\left(1+\frac{\ep_\psi \varphi(q)}{\sqrt{y}}\right)
\end{split}
\end{align}
and 
\begin{align}
\begin{split}
\theta(y; q, 1)+\theta(y; q, \ell_0)&\geq \frac{2y}{\varphi(q)}\left(1-\frac{\ep_\theta\varphi(q)}{\sqrt{y}}\right),
\end{split}
\end{align}
where $\ep_\psi$ and $\ep_\theta$ are given by 
\begin{center}
\begin{tabular}{|c||c|c|c|c|c|c|c|c|c|c|} \hline
$q$ & $5$ & $7$ & $8$ & $9$ & $12$ & $16$ & $24$ & $11\leq q\leq 23$  \\ \hline \hline
$\ep_\psi$ & $.807$ & $.78$ & $.927$ & $.789$ & $.863$ & $.774$ & $.745$ & $.912$  \\ \hline \hline
$\ep_\theta$ & $1.413$ & $1.106$ & $1.5$ & $1.11$ & $1.5$ & $1.03$ & $1.5$ & $1.1$   \\ \hline \hline 
\end{tabular}
\end{center}
Further,
\begin{align*}
\frac{y}{\varphi(24)}\left(1-\frac{\varphi(24)}{\sqrt{y}}\right) \leq 
\theta(y; 24, 5)\leq \psi(y; 24, 5)\leq \frac{y}{\varphi(q)}\left(1+\frac{0.745\varphi(24)}{\sqrt{y}}\right). 
\end{align*}

\end{lemma}

As a consequence, we have the following result.

\begin{lemma}\label{pi<30}
Let $q\in \{8, 9, 12, 16, 24\}$ or $5\leq q\leq 23$ be a prime and  $\ell_0$ be an integer coprime to $q$ with 
$\ell_0 \not\equiv 1\pmod q$. Then for $y\geq 1500$, we have 
\begin{align}\label{m-+1}
\begin{split}
&\sum_{l=1, \ell_0}\left(\psi(2y; q, l)-\theta(y; q, l)+\theta\left(\frac{2y}{3}; q, l\right)-\theta\left(\frac{y}{2}; q, l\right)
+\theta\left(\frac{2y}{5}; q, l\right)\right)\\
\leq &\frac{y}{\varphi(q)}\left\{\frac{47}{15}+\frac{2\sqrt{2}\ep_\psi}{\sqrt{y}}\left(1+\frac{1}{\sqrt{3}}+\frac{1}{\sqrt{5}}\right)
+\frac{2\ep_\theta}{\sqrt{y}}\left(1+\frac{1}{\sqrt{2}}\right)\right\},
\end{split}
\end{align}
where $\ep_\psi$ and $\ep_\theta$ are given in Lemma \ref{pi<24}. 
Also for each $y\geq 15$, there is a prime $p\equiv 5\pmod {24}$ with  $y+1<p\leq 2y$. Further, for $y\geq 6$, there 
is a  prime $p\equiv \pm 5\pmod 8$ with $y+1<p\leq 2y$. And for $y\geq 9$, there 
is a  prime $p\equiv 5\pmod 12$ with $y+1<p\leq 2y$. 
\end{lemma}

\begin{proof}
The assertion \eqref{m-+1} is immediate from Lemma \ref{pi<24} and using the inequality $\theta(y; q, l)\leq \psi(y; q, l)$ valid for all 
$y$.  For primes $p\equiv 5\pmod {24}$, again from Lemma \ref{pi<24}, we have 
\begin{align*}
\theta(2y; 24, 5)&-\theta(y+1; 24, 5)\geq \frac{2y}{\varphi(24)}\left(1-\frac{\varphi(24)}{\sqrt{2y}}\right)\\
& -\frac{y+1}{\varphi(24)}\left(1+\frac{0.745\varphi(24)}{\sqrt{y+1}}\right)\\
&=\frac{y}{\varphi(24)}\left\{1-\frac{2\varphi(24)}{\sqrt{2y}}-\frac{1}{y}-
\frac{0.745\varphi(q)}{\sqrt{y+1}}-\frac{0.745\varphi(q)}{y\sqrt{y+1}}\right\}>0,
\end{align*}
for $y\geq 400$.  Thus, there is a prime $p\equiv 5\pmod {24}$ with $y+1<p\leq 2y$ for $y\geq 400$.  This is also true for 
$15\leq y<400$ by checking at integer values of $y$. Since a prime congruent to $ 5\pmod {24}$ is also congruent to 
$ 5\pmod 8$ and $5\pmod{12}$, the last two assertions can be obtained by checking it in the range $6\leq y<15$.  
\end{proof}

\iffalse 

The following result follows from \cite{Lehmer}. 

\begin{lemma}\label{lehmer}
The equation $X(X^2-3)=2^{l_1}3^{l_2}11^{l_3}13^{l_4}$ in positive integers $X, l_1, l_2, l_3, l_4$ 
 does not have any solution. 
\end{lemma}
\begin{proof}
Since $3$ divides the right hand side of the given equation, we have $3\mid X(X^2-3)$ implying $3\mid X$. Put $X=3X_1$. Then 
$$
2^{l_1}3^{l_2}11^{l_3}13^{l_4}=X(X^2-3)=3\cdot (3X_1)(3X^2_1-1).
$$ 
This shows that  with $N:=3X_1^2$, we have that
$N(N-1)$ is composed of primes from the set $\{2, 3, 11, 13\}$.  Then $N$ is given by
\cite[Table IA]{Lehmer} with $t=6$ which is possible only for $N=27$. But then $3X^2_1=27$ gives 
$X=3X_1=9$ and $9\cdot (9^2-3)=2\cdot 3^3\cdot 13$ which is not a solution of the original equation. 
Hence, the assertion.
 \end{proof}
\fi 

In the next section, we use use Lemmas \ref{m14}, \ref{pix} and \ref{pi<30} to obtain upper bound for prime powers dividing a product of 
Catalan numbers and middle binomial coefficients. 

\section{Upper bound for prime powers dividing a product of Catalan numbers and middle binomial coefficients}

For positive integers $1<m_1\leq m_2\leq \cdots \leq m_k$, let 
\begin{align*}
\mathcal{D}:=\mathcal{D}(m_1, m_2, \ldots, m_k):=\prod^k_{i=1}D_{m_i}, \quad D_{m_i}\in \{C_{m_i}, B_{m_i}\}. 
\end{align*}
Let $n_0$ be a positive integer. Recall the definition of $M_{n_0}(\ell)$ given in \eqref{Mn}. 
We use analytic methods to find an upper bound for 
\begin{equation*}
M_{n_0}(\mathcal{D}):=\log\left(\prod_{\substack{p^{\nu_p}\|\cD \\ 
p\equiv \pm 1\pmod {n_0}}} p^{\nu_p}\right)=
\sum _{\substack{p^{\alpha_p}\|\cD \\ p\equiv \pm 1\pmod {n_0}}}\nu_p\log p.
\end{equation*}
This is the content of the following lemma. 

\begin{lemma}\label{MnBD}
For $n_0\geq 25$, we have 
\begin{equation}\label{Mnub}
M_{n_0}(\cD)\leq  \begin{cases}
\left(\frac{3.9}{\varphi(n_0)}+2.92\frac{\log 3n_0}{n_0}\right)(\log \cD-\log 2), 
& \  {\rm if} \ n_0 \ {\rm is \ even};\\
\left(\frac{3.9}{\varphi(n_0)}+1.46\frac{\log 3n_0}{n_0}\right)(\log \cD-\log 2),   & \ 
{\rm if} \ n_0 \ {\rm is \ odd}.
\end{cases}
\end{equation}
Let $n_0\in \{9, 16, 24\}$ or $5\leq n_0\leq 23$ be a prime.  We have  
\begin{align}\label{Mnub<30}
M_{n_0}(\cD)\le \begin{cases}
 \frac{\delta_0}{\varphi(n_0)}\log \cD, & {\rm if} \quad m_k<1500;\\
\frac{\delta_0}{\varphi(n_0)}\left(\log \cD -\log 2\right), & {\rm if} \quad m_k\geq 1500; \\
\end{cases}
\end{align}
where $\delta_0$ is given by 
\begin{center}
\begin{tabular}{|c||c|c|c|c|c|c|c|} \hline
$n_0$ & $5$ & $7$  & $9$ & $16$ & $24$ & $11\leq n_0\leq 23$    \\ \hline \hline
$\delta_0$ & $2.61$ & $3.19$ & $3.57$ & $2.89$ & $2.746$ & $3.3$   \\ \hline 
\end{tabular}
\end{center}
\end{lemma}
 
\begin{proof}
Let $t_{1j}$ and $t_{2j}$ be the number of $i'$s such that $D_{m_i}=C_j$ and $D_{m_i}=B_j$, 
respectively. Put $t_j=t_{1j}+t_{2j}$. Then  
$$\log \cD=\sum_{1\leq i \leq k}\log D_{m_i}=\sum_{1<j\leq m_k}(t_{1j}\log C_j+t_{2j}\log B_j)\geq \sum_{1<j\leq m_k}t_j\log C_j$$
since $B_m>C_m$. Let $7<M\leq m_k$ be an integer which we will choose later on. Using 
Lemma \ref{m14}, we get
\begin{align*}
\log \cD&\geq \sum_{j\leq M}t_j\log C_j+ \sum_{j>M}t_j\log C_j\\
&\geq \log 2+\sum_{j\leq M}t_j\log ({C_j}/{2})+\frac{\log ({C_M}/{2})}{M}\sum_{j>M}t_jj,
\end{align*}
so that 
\begin{align}\label{cDbd}
\sum_{j>M}t_jj\leq \frac{M}{\log ({C_M}/{2})}\left(\log \cD-\log 2-\sum_{j\leq M}t_j\log ({C_j}/{2})\right). 
\end{align}
Here, as usual, the empty sum is taken to be $0$. 
 For a prime number $p$ and a positive integer $t$, we write $\nu_p(t)$ for the exact exponent of 
$p$ in the prime factorization of $t$. Given a positive integer $j$, let 
\begin{align*}
\xi_1(j): = \sum _{ p\equiv \pm 1\pmod {n_0}}\nu_p(C_j)\log p \quad {\rm and} \quad 
\xi_2(j): = \sum _{ p\equiv \pm 1\pmod {n_0}}\nu_p(B_j)\log p.
\end{align*}
Then $\xi_1(j)\leq \xi_2(j)$ and hence $M_{n_0}(\mathcal{D})\leq \sum_{j}t_j\xi_2(j)$.   For a prime $p$, we have 
\begin{align*}
\nu_p(B_j)=\sum_{\ell\ge 1}\left( \left\lfloor \frac{2j}{p^\ell}\right\rfloor-2\left\lfloor \frac{j}{p^\ell}\right\rfloor\right)
\leq 
\begin{cases}
1, & {\rm if} \ \frac{2j}{2i}<p\leq \frac{2j}{2i-1},~ i\in \{1, 2\};\\
0, & {\rm if} \ \dfrac{2j}{2i+1}<p\leq \frac{2j}{2i},~ i\in \{1, 2\};\\
\left\lfloor \frac{\log(2j)}{\log p}\right\rfloor, & {\rm if} \ p\leq \dfrac{2j}{5}.
\end{cases}
\end{align*}
Therefore, 
\begin{align}\label{xim}
\begin{split}
\xi_2(j)&\leq \sum _{\substack{(2j)^{1/2}<p\leq 2j \\ p\equiv \pm 1\pmod {n_0}}}
\left( \left\lfloor \frac{2j}{p}\right\rfloor-2\left\lfloor \frac{j}{p}\right\rfloor\right)\log p 
+ \sum _{\substack{p\leq (2j)^{1/2}\\ p\equiv \pm 1\pmod {n_0}}}\left\lfloor \frac{\log (2j)}{\log p}\right\rfloor \log p\\
&\leq \sum _{\substack{p\leq 2j \\ p\equiv \pm 1\pmod {n_0}}}\left\lfloor \frac{\log (2j)}{\log p}\right\rfloor \log p-
\sum _{\substack{1\le i\le 2\\ \frac{2j}{2i-1}<p\leq \frac{2j}{2i} \\ p\equiv \pm 1\pmod {n_0}}}\log p \\
& \leq \sum_{\ell\in \{1, -1\}}\left\{\psi(2j; n_0, \ell)+\sum^3_{t=2}\theta(2j/t; n_0, \ell)-\sum^2_{t=1}\theta(j/t; n_0, \ell)\right\}. 
\end{split}
\end{align} 
Recall that $\pi(x; n_0, \ell)$ stands for the number of primes $p\le x$ satisfying the congruence $p\equiv \ell\pmod {n_0}$. 
We put $\pi_{\pm 1}(x):=\pi(x; n_0,1)+\pi(x; n_0, -1)$. Then 
\begin{align}\label{2mm2m/3}
\xi_2(j)\leq (\pi_{\pm 1}(2j)-\pi_{\pm 1}(j)+\pi_{\pm 1}(2j/3))\log (2j),
\end{align}
by \eqref{xim}.  
Let us assume that $n_0\geq 25$. Let $t>0$. We split the analysis in two cases according to whether $2j\leq (3n_0)^{1+1/t}$ or 
$2j> (3n_0)^{1+1/t}$. 

Assume first that $2j\ge (3n_0)^{1+1/t}$. Then $2j/3n_0\ge (2j)^{1/(1+t)}$ and 
therefore 
$$\log(j/n_0)\ge \log(2j/3n_0)\ge (\log (2j))/(1+t).
$$  
From \eqref{2mm2m/3} and Lemma \ref{pix}, we get 
\begin{align*}
\xi_2(j)&\leq  \frac{4j\log 2j}{\varphi(n_0)\log (j/n_0)}+\frac{(4j/3)\log 2j}{\varphi(n_0)\log (2j/3n_0)}
\leq \frac{16(1+t)j}{3\varphi(n_0)}. 
\end{align*} 
In the smaller range $2j\leq (3n_0)^{1+1/t}$, using the trivial estimates and the fact that primes congruent 
to one of $\pm 1$ modulo $n_0$ are of the form $2ln_0\pm 1$ when $n_0$ is odd, we get  
\begin{align*}
&\pi_{1, -1}(2j)-\pi_{1, -1}(j)+\pi_{1, -1}(2j/3)\leq \pi_{1, -1}(2j)\\
\leq &\begin{cases}
\frac{2j-1}{n_0}+\frac{2j+1}{n_0}=\frac{4j}{n_0}, & {\rm if} \ n_0 \ {\rm is \ even};\\
\frac{2j-1}{2n_0}+\frac{2j+1}{2n_0}=\frac{2j}{n_0}, & {\rm if} \ n_0 \ {\rm is \ odd}.
\end{cases}
\end{align*}
Let $\eta:=1, 2$ according to whether $n_0$ is even or odd, respectively. From \eqref{2mm2m/3}, we get 
$$\xi_2(j)\leq \left(\frac{4}{\eta}\right)\frac{j\log 2j}{n_0}.$$ 
We choose 
$$t:=\frac{3.0003}{4\eta}\frac{\varphi(n_0)\log 3n_0}{n_0} \qquad {\rm and} \qquad 
M:=\left\lfloor \frac{1}{2}(3n_0)^{1+\frac{1}{t}}\right\rfloor. 
$$
Since $n_0/\varphi(n_0)\geq 2/\eta$, we observe that 
$$
\frac{1}{2}(3n_0)^{1+\frac{1}{t}}\geq 1.5n_0\exp\left(\frac{8}{3.0003}\right)>686 \quad {\rm for} \quad n_0\geq 32,
$$
which together with $M\geq 686$ for each $25\leq n_0<32$ implies $M\geq 686$ for all $n_0\geq 25$.  
From \eqref{cDbd}, we have 
\begin{align*}
&M_{n_0}(\cD)\le \sum_{j}t_j\xi_2(j) \\
\leq & \frac{M}{\log (\frac{C_M}{2})}\frac{16(1+t)}{3\varphi(n_0)}
\left(\log \frac{\cD}{2}-\sum_{j\leq M}t_j\log \left(\frac{C_j}{2}\right)\right)+ \sum_{j\leq M}\frac{4t_jj\log 2j}{\eta n_0}\\
\leq&  \frac{M}{\log (\frac{C_M}{2})}\frac{16(1+t)}{3\varphi(n_0)}\log  \frac{\cD}{2}-
\sum_{j\leq M}t_j\left(\frac{M\log ({C_j}/{2})}{\log ({C_M}/{2})}\frac{16(1+t)}{3\varphi(n_0)}-\frac{4j\log 2j}{\eta n_0} \right).
\end{align*}
Since $M>7$,  and we get from \eqref{0001} and $\log 2M\leq (1+1/t)\log 3n_0$
that 
\begin{align*}
&\frac{M\log ({C_j}/{2})}{\log ({C_M}/{2})}\frac{16(1+t)}{3\varphi(n_0)}-\frac{4j\log 2j}{\eta n_0} \\
=& \frac{4M\log ({C_j}/{2})}{n_0\log ({C_M}/{2})}\left(\frac{4(1+t)n_0}{3\varphi(n_0)}-
\frac{j\log 2j}{\eta \log ({C_j}/{2})}\frac{\log ({C_M}/{2})}{M}\right)\\
\geq & \frac{16M\log ({C_j}/{2})}{3n_0\log ({C_M}/{2})}\frac{n_0}{\varphi(n_0)}
\left(\frac{(1+t)n_0}{\varphi(n_0)}-\frac{3\varphi(n_0)}{4\eta n_0}
\frac{1.0001(1+t)\log 3n_0}{t} \right)\geq 0,
\end{align*}
since 
$$
t=\frac{3.0003}{4\eta}\frac{\varphi(n_0)\log 3n_0}{n_0}.
$$ 
Therefore, we have from $M\geq 686$ and Lemma \ref{m14}, that 
\begin{align*}
M_{n_0}(\cD)&\leq  \frac{M}{\log ({C_M}/{2})}\frac{16(1+t)}{3\varphi(n_0)}\log \frac{\cD}{2}\\
&\leq  \frac{16}{3}\frac{686}{\log ({C_{686}}/{2})}
\left(\frac{1}{\varphi(n_0)}+3.0003\left(\frac{\log 3n_0}{4\eta n_0}\right)\right)(\log \cD-\log 2),
\end{align*}
which gives the assertion \eqref{Mnub}. 

We now consider $n_0\leq 24$ as given in the statement of the lemma. Then either $n_0\in \{9, 16, 24\}$, or 
$n_0$ is a prime with $5\leq n_0\leq 23$. We check with exact computations that for $j\leq 1500$,  
\begin{align*}
\xi_1(j) \leq \frac{\delta_0\log C_j}{\varphi(n_0)} \quad {\rm and} \quad \xi_2(j) \leq \frac{\delta_0\log B_j}{\varphi(n_0)},
\end{align*}
where $\delta_0$ are given in the statement of the lemma.  Hence, we have 
$$
M_{n_0}(\cD)=\sum_{1<j\leq m_k}(t_{1j}\xi_1(j)+t_{2j}\xi_2(j))\leq \frac{\delta_0}{\varphi(n_0)}
\sum_{1<j\leq m_k}(t_{1j}\log C_j+t_{2j}\log B_j),$$ 
which gives the assertion \eqref{Mnub<30} for $m_k<1500$. 

We now take $m_k\geq 1500$.  From \eqref{xim} and Lemma \ref{pi<30}, we get  
$$
\xi_1(j)\leq \xi_2(j)\leq \frac{\delta_{1}j}{\varphi(n_0)}\leq \frac{\delta_1j}{\log (C_j/2)}\frac{\log (D_j/2)}{\varphi(n_0)}
\quad {\rm   for} \quad  j\geq 1500,
$$ 
where 
\begin{align*}
\delta_{1}&=\frac{47}{15}+\frac{2\sqrt{2}\ep_\psi}{\sqrt{1500}}\left(1+\frac{1}{\sqrt{3}}+\frac{1}{\sqrt{5}}\right)
+\frac{2\ep_\theta}{\sqrt{1500}}\left(1+\frac{1}{\sqrt{2}}\right),
\end{align*}
and $\ep_\psi$ and $\ep_\theta$ are given in Lemma \ref{pi<24}.  By Lemma \ref{m14},  we have 
$$
\frac{\delta_1j}{\log (C_j/2)}\leq \frac{\delta_1}{1.37}\quad {\text{\rm for each}}\quad j\geq 1500,
$$ 
and we find that 
${\delta_1}/{1.37}\leq \delta_0$.  Thus, 
$$
\xi_1(j)\leq \xi_2(j)\leq \frac{\delta_{0}\log (D_j/2)}{\varphi(n_0)}\quad {\rm  for \ each} \quad j\geq 1500,
$$ 
and  therefore 
\begin{align*}
&M_{n_0}(\cD)= \sum_{j<1500}(t_{1j}\xi_1(j)+t_{2j}\xi_2(j))+
\sum_{j\geq 1500}(t_{1j}\xi_1(j)+t_{2j}\xi_2(j))\\
\leq & \frac{\delta_0}{\varphi(n_0)}\sum_{j<1500}(t_{1j}\log C_j+t_{2j}\log B_j)+
 \frac{\delta_0}{\varphi(n_0)}\sum_{j\geq 1500}\left(t_{1j}\log \frac{C_j}{2}+t_{2j}\log \frac{B_j}{2}\right)\\
\leq &\frac{\delta_{0}}{\varphi(n_0)}\left(\log \cD-\log 2\right).
\end{align*}
Hence, the assertion  \eqref{Mnub<30} follows and the proof is complete. 
\end{proof}

\section{Proof of Theorem \ref{theorem1} \label{th1}}

We recall that for $n\geq 0$
\begin{align*}
U_n=\frac{\alpha^n-\beta^n}{\alpha-\beta} \quad {\rm and} \quad 
V_n= \alpha^n+\beta^n.
\end{align*}
where $\alpha$ and $\beta$ are the roots of the quadratic equation $\lambda^2-r\lambda-s=0$ and 
$r, s$ are coprime nonzero integers with $r^2+4s\neq 0$. We suppose that $\alpha/\beta$ is not a root of unity.  We also recall that we assume
that $r>0$. When $\alpha,\beta$ are real, these conventions imply that $\alpha$ is positive 
so $\alpha>|\beta|$ and in this case $U_n>0$ and $V_n>0$ for all $n\ge 1$.  Further, we put $x={\be}/{\al}$. Thus, $|x|\le 1$. 

Note that $U_1=1\in {\mathcal PBC}$. In fact, if $U_n=\pm 1$ (or $V_n=\pm 1$) then $U_n$ (or $V_n$) are 
also in ${\mathcal PBC}$.  The equations $U_n=\pm 1$ and $V_n=\pm 1$ are important from the Diophantine point of view. 
However, such equations have been solved completely and we refer to \cite{BHV} for more details.  For this reason, 
whenever we study the equations \eqref{eqn1} and \eqref{eqn2}, we omit the cases $n=1$, $U_n=\pm 1$ and 
$V_n=\pm 1$. Thus, we also assume that $m_1>1$. 

We first treat the case of the sequence $\{U_n\}_{n\ge 0}$.  
Assume that the equation \eqref{eqn1} has a solution. Then 
$$
|U_n|=\cD=D_{m_1}\cdots D_{m_k}, \quad D_{m_i}\in \{C_{m_i}, B_{m_i}\}. .
$$ 
For a divisor $n_0$ of $n$, we will compare the upper bound of $M_{n_0}(\cD)$ given by Lemma 
\ref{MnBD} with a lower bound on it obtained by using Lemma \ref{primn_0}. 
We will choose a suitable divisor $n_0$ of $n$ such that these bounds contradict each 
other and hence for $n$ with such divisors $n_0$, $|U_n|$ cannot be a product of Catalan numbers and 
middle binomial coefficients.  

Recall that a prime $p\mid U_n$ is a primitive divisor of $U_n$ if 
$p\nmid U_t$ for $t<n$ and $p\nmid r^2+4s$. Further, the primitive prime 
divisors of $U_n$ are congruent to one of $ \pm 1$ modulo $n$. From the well known result from
\cite{BHV}, we know that a primitive divisor for $U_n$ exist for all $n>30$. Further, for 
$5\leq n\leq 30$, $n\neq 6$, the pairs $(r, s)$ for which a primitive 
divisor for $U_n$ does not exist are given by 
\begin{center}
\begin{tabular}{|c|c|} \hline
$n$ & $(r, s)$   \\ \hline
$5$ & $(1, 1), (1, -2), (1, -3), (1, -4),(2, -11), 
(12, -55), (12, -377)$  \\ \hline
$7$ & $(1, -2), (1, -5)$ \\ \hline
$8$ & $(1, -2), (2, -7)$ \\ \hline
$10$ & $(2, -3), (5, -7), (5, -18)$ \\ \hline
$12$ & $(1, 1), (1, -2), (1, -3), (1, -4), (1, -5), (2, -15)$ \\ \hline
$13, 18, 30$ & $ (1, -2)$ \\ \hline
\end{tabular}
\end{center}
We checked that for $(r, s)$ given above with $n\geq 5, n\neq 6$, the equation \eqref{eqn1} holds 
in several instances. The roots $(\alpha,\beta)$ are real only when $(r,s)=(1,1)$ and then
$$
(r, s, n)=(1,1,5),~(1, 1, 12),\quad U_5=C_3,~U_{12}=B_1^6C_2^2=B_1^2 B_2^2.
$$ 
Hence, we assume now that $U_n$  has a primitive prime divisor $p$ and so
$p\equiv \pm 1\pmod n$. Let $P_n:=P(U_n)$ be the largest primitive divisor of $U_n$. 
From \eqref{eqn1}, we have that $P_n\mid B_{m_k}$ and so $2m_k\geq P_n+1$ since  
$P_n$ is odd. Let $Q_n$ be the least prime congruent to one of $\pm 1$ modulo $n$. Then 
$2m_k\geq P_n+1\geq Q_n+1$ and therefore
\begin{align}\label{logal}
2|\alpha|^{n}\ge \left|\frac{\alpha^n-\beta^n}{\alpha-\beta}\right|=|U_n|\ge C_{m_k}.
\end{align}
From Lemma \ref{m14},  we have 
\begin{align}\label{logal1}
n\log |\al|\geq \log (C_{m_k}/2)\geq \begin{cases}
1.36m_k \ge 0.68(Q_n+1)\geq 0.68n, & n\geq 400;\\
1.38m_k \ge 0.69(Q_n+1)\geq 0.69n, & n\geq 4200,
\end{cases}
\end{align}
since $Q_n\geq n-1$. 

We have 
\begin{align*}
\log \cD\leq \log |U_n|\leq n\log |\al|+\log |1-x^n|.
\end{align*}
Now we complete the proof by choosing suitable $n_0$ and comparing upper and lower 
bounds of $M_{n_0}(U_n)=\cM_{n_0}(\cD)$. For $n_0\in \{9, 16, 24\}$, or $n_0$ an odd prime 
power, we define 
\begin{align}\label{gn}
g(n_0):=\begin{cases}
\frac{\delta_0}{\varphi(n_0)}, &{\rm if} \ n_0\in \{9, 16, 24\}, \ {\rm or} \ n=p\leq 23;\\
\frac{3.9}{\varphi(n_0)}+\frac{1.46\log 3n_0}{n_0}, & {\rm if} \ n_0\ge 25 \ {\rm is \ odd},
\end{cases}
\end{align} 
where $\delta_0$ is stated in Lemma \ref{MnBD}. By Lemma \ref{MnBD}, we have 
\begin{align}\label{Dub}
M_{n_0}(\cD)\leq g(n_0)\log |U_n|\leq g(n_0)\left(n\log |\al|+\log |1-x^n|\right). 
\end{align}
Let $p^{h+t}\mid n$, where $p$ is a prime and $h>0, t\geq 0$ are integers such that 
$p^h>4$.  Taking $n_0=p^h$ and using \eqref{n_0U} in Lemma \ref{primn_0}, 
we get a lower bound for $M_{n_0}(U_n)=M_{n_0}(\cD)$ which we compare with \eqref{Dub}. We obtain 
\begin{align*}
&g(n_0)\left(n\log |\al|+\log |1-x^n|\right)\\
\geq &
\left(1-\frac{1}{p^{t+1}}\right)n\log |\al|+\log |1-x^n|-\log |1-x^{n/p^{t+1}}|-\log(p^{t+1}),
\end{align*}
implying 
\begin{align}\label{Dcmp}
&\left(1-\frac{1}{p^{t+1}}-g(n_0)\right)
\leq & \frac{(g(n_0)-1)\log |1-x^n|+\log |1-x^{n/p^{t+1}}|+\log(p^{t+1})}{n\log |\al|}.
\end{align}

We consider different cases. 

\subsection{The case when $n$ is even }

We assume that $n>720$. We choose $n_0=p^h$ and $t$ as follows: 
\begin{align}\label{noeven}
(n_0, t)\in \{(2^4, 1), (3^2, 1), (5, 1)\}\cup \{(p, 0) : p>5\}.
\end{align}
Since $2^4\cdot 3^2\cdot 5=720$, we find that for each even 
$n>720$, there is some $(n_0, t)$ in \eqref{noeven}  with $n_0p^t\mid n$. From the triangle inequality 
$$
2|\al^{n/2}|\leq |\al^{n/2}-\be^{n/2}|+|\al^{n/2}+\beta^{n/2}|,
$$  
we have either 
$|\al^{n/2}-\be^{n/2}|\geq |\al|^{\frac{n}{2}}$,  or $|\al^{n/2}+\be^{n/2}|\geq |\al|^{\frac{n}{2}} $.  
Therefore, 
\begin{align*}
 &|\al^{n}-\be^{n}|=|\al^{n/2}-\be^{n/2}||\al^{n/2}+\be^{n/2}|\\
\geq &\begin{cases}
|\al|^{\frac{n}{2}}|\al^{\frac{n}{2}}+\be^{\frac{n}{2}}|=|\al|^{\frac{n}{2}}|V_{\frac{n}{2}}|\geq 
|\al|^{\frac{n}{2}}, \quad & \quad  |\al^{\frac{n}{2}}-\be^{\frac{n}{2}}|\geq |\al|^{\frac{n}{2}};\\
|\al-\be||\frac{\al^{\frac{n}{2}}-\be^{\frac{n}{2}}}{\al-\be}||\al|^{\frac{n}{2}}\geq |U_{\frac{n}{2}}||\al|^{\frac{n}{2}},
\geq |\al|^{\frac{n}{2}} & \quad  |\al^{\frac{n}{2}}+\be^{\frac{n}{2}}|\geq |\al|^{\frac{n}{2}},\\
\end{cases}
\end{align*}
since $V_{{n}/{2}}, U_{{n}/{2}}$ are integers and $|\al-\be|\geq 1$. Hence,  
$$|1-x^n|=|\al|^{-n}|\al^n-\be^n|\geq |\al|^{-\frac{n}{2}}.$$
Using the above inequality together with the inequality $|1-x^{n/p^{t+1}}|\leq 2$ (since $|x|\leq 1$) in \eqref{Dcmp}, we get 
\begin{align*}
\frac{\log (2p^{t+1})}{n\log |\al|}&\geq 1-\frac{1}{p^{t+1}}-g(n_0)+\frac{g(n_0)-1}{2}=\frac{1}{2}-\frac{1}{p^{t+1}}-\frac{g(n_0)}{2}. 
\end{align*}
From \eqref{logal1}, we have 
\begin{align*}
0\geq \frac{1}{2}-\frac{1}{p^{t+1}}-\frac{g(n_0)}{2}-\frac{\log (2p^{t+1})}{0.68n}. 
\end{align*}
For a fixed choice of $n_0=p^h$ and $t$, the right--hand side of the above inequality is an increasing function  of 
$n$. We check that for $(n_0, t)$ in \eqref{noeven} with $n_0<29$, the above inequality is not valid at $n=720$ 
and hence it is not valid for any $n\geq 720$. Further, for $n_0\geq 29$, we have $n_0=p$ is prime, which together 
with the observation that $g(p)$ is a decreasing function of $p$, we obtain 
\begin{align*}
0\geq \frac{1}{2}-\frac{1}{p^{t+1}}-\frac{g(n_0)}{2}-\frac{\log (2p^{t+1})}{0.65n}\geq 
 \frac{1}{2}-\frac{1}{29}-\frac{g(29)}{2}-\frac{\log (2\times 29)}{0.68n}. 
\end{align*}
We check that the right--most side is positive for $n\geq 720$ and hence we get  a contradiction for all 
$n\geq 720$. Thus, equation \eqref{eqn1} has no even solution $n>720$. 

\subsection{The case when $\al, \be$ are complex conjugates}

From the previous section, we may assume that either $n>720$ is odd or $n$ is an even number $\leq 720$.  
Since we are shooting for the inequality $n<6500$, we may assume that $n\geq 6500$ is odd.  Also, we have $Q_n\geq 2n-1$ which together with 
$2m_k\geq Q_n+1$, inequality \eqref{logal} and Lemma \ref{m14} gives 
$$\log |U_n|\geq 1.38n \qquad {\rm and} \qquad n\log |\al|\geq 1.38n.$$ 
We choose $n_0$ of the form $p^h$ and $t$ given by  
\begin{align}\label{nocomp}
(n_0, t)\in \{(3^2, 2), (5, 1), (7, 1)\}\cup \{(p, 0) : p\geq 11\}.
\end{align}
Since $3^3\cdot 5^2\cdot 7<6500$, we find that for each odd $n\geq 6500$, there is some 
$(n_0, t)$ in \eqref{nocomp}  with $n_0p^t\mid n$. 

First we consider the case when $\log |\al|\leq 4$.  We use  
$$|1-x^n|=\frac{|\al-\be||U_n|}{|\al|^n}\geq \frac{|U_n|}{|\al|^n}, \quad |1-x^{n/p^{t+1}}|\leq 2 \quad {\rm and} 
\quad \log |\al|\leq 4$$ 
in \eqref{n_0U} and compare it with \eqref{Dub} to obtain 
\begin{align*}
g(n_0)\log |U_n|\geq M_{n_0}(U_n) \geq \log |U_n|-\frac{4n}{p^{t+1}}-\log (2p^{t+1}). 
\end{align*}
Since $\log |U_n|\geq 1.38n$, we obtain 
\begin{align}\label{al<4U}
0\geq 1.38(1-g(n_0))-\frac{4}{p^{t+1}}-\frac{\log (2p^{t+1})}{n}.  
\end{align}
For a fixed choice of $n_0=p^h$ and $t$, the right--hand side of the above inequality is an increasing function  of 
$n$. We check that for $(n_0, t)$ in \eqref{nocomp} with $n_0<29$, the above inequality 
is not valid at $n=6500$ and hence it is not valid for any $n\geq 6500$. Further, for $n_0\geq 29$, we have $n_0=p$ is 
prime and $t=0$, which together with the observation that $g(p)$ is a decreasing function of $p$, we obtain 
\begin{align*}
0 & \geq 1.38(1-g(n_0))-\frac{4}{p^{t+1}}-\frac{\log (2p^{t+1})}{n}\\
& \geq 1.38(1-g(29))-\frac{4}{29}-\frac{\log (2\cdot 29)}{n}.  
\end{align*}
We check that the right--most side is positive for $n\geq 6500$ and hence we get  a contradiction for any 
$n\geq 6500$. Thus, the equation \eqref{eqn1} does not have an odd solution $n\geq 6500$ in case $\log |\alpha|<4$. 

Assume now that $\log |\al|>4$.  By Lemma \ref{-44.72}, we get 
$$
\log |1-x^n|\geq -f(n)\log |\al|,
$$ 
where $f(n)$ is given by formula \eqref{f(x)}. Using this inequality along with 
$$|1-x^n|\leq 2\quad {\text{\rm  and}}\quad n\log |\al|\geq 4n
$$ 
(since $\log |\al|>4$)  in \eqref{Dcmp}, we obtain
\begin{align}\label{al>4U}
0\geq 1-\frac{1}{p^{t+1}}-g(n_0)+\frac{(1-g(n_0))f(n)}{n}-\frac{\log (2p^{t+1})}{4n}. 
\end{align}
For a fixed $n_0=p^h$ and $t$, the right--hand side of the above inequality is an increasing function  of 
$n$. We check that for $(n_0, t)$ in \eqref{nocomp} with $n_0<29$, the above inequality 
is not valid at $n=6500$ and hence it is not valid for any $n\geq 6500$. Further, for $n_0\geq 29$, we have 
$n_0=p$ is prime and $t=0$ and hence the right--hand side of the above inequality is at least 
\begin{align*}
1-\frac{1}{29}-g(29)+\frac{(1-g(29))f(n)}{n}-\frac{\log (2\cdot 29)}{4n }. 
\end{align*}
We check that the above quantity is positive for $n\geq 6500$ and hence we get  a contradiction for any 
$n\geq 6500$. Thus, the equation \eqref{eqn1} has no odd solution $n\geq 6500$ in case $\log |\alpha|>4$. 

\subsection{The case when $\al, \be$ are real and $n\geq 5, n\notin \{6, 8, 12, 24\}$}

We now consider the case when $\alpha$ and $\beta$ are real.  Recall that in this case $\alpha>0$ and $U_n>0$. For the proof of  Theorem \ref{theorem1}, we 
may assume that $n\geq 5$, $n\notin \{6, 8, 12, 24\}$. We will consider the case $n=24$ separately in the 
next section. We choose $n_0=p^h$ with $t=0$ as   
\begin{align}\label{noreal}
n_0\in \{2^4, 3^2\}\cup \{p : p\geq 5\}.
\end{align}
Note that each $n\geq 5, n\notin \{6, 8, 12, 24\}$ is divisible by some $n_0$ in \eqref{noreal}.   

Let $n_0=2^4=16$. Then $p=2$, $4\mid n$ and hence 
 $$g(16)\log |1-x^n|<0  \quad {\rm and} \quad \frac{1-x^n}{1-x^{n/p}}=1+\sum^{p-1}_{i=1}x^{in/p}>1.$$ 
Using this in \eqref{Dcmp} together with $n\geq 16$ and $\al\geq {\displaystyle{\frac{1+\sqrt{5}}{2}}}$, we get 
\begin{align*}
0 \geq 1-\frac{1}{2}-g(16)-\frac{\log 2}{n\log \al} \geq 
\frac{1}{2}-g(16)-\frac{\log 2}{16\log \left(\frac{1+\sqrt{5}}{2}\right)}.\end{align*}
We find that the right--most quantity is positive, which is a contradiction. 
Thus, equation \eqref{eqn1} has no solution when $\alpha,~\beta$ are real with $16\mid n$. 

Let $n_0\neq 2^4$. Then $p>2$. Writing 
$$
1-x^n=(1-x^{n/p})\left(\frac{1-x^n}{1-x^{n/p}}\right),
$$ 
we have  
 \begin{align*}
 |1-x^{n/p}|\leq 2 \quad {\rm and}  \quad \frac{1-x^n}{1-x^{n/p}}=\begin{cases}
 1+\sum^{p-1}_{i=1}y^{i}>1, &  y=x^{\frac{n}{p}}>0;\\
 \frac{1-y(y^{(p-1)/2})^2}{1-y}\geq \frac{1}{1-y}>\frac{1}{2}, &  y=x^{\frac{n}{p}}<0.
\end{cases}
 \end{align*}
 Using this in \eqref{Dcmp}, we obtain 
 \begin{align*}
\log \al &\leq  \frac{(g(n_0)-1)\log \left(\frac{1-x^n}{1-x^{n/p}}\right)+g(n_0)\log (1-x^{n/p})+\log p}{
n(1-{1}/{p}-g(n_0)))}\\
 &\leq \frac{(1-g(n_0))\log 2+g(n_0)\log 2+\log p}{n(1-{1}/{p}-g(n_0))}\\
&=  \frac{\log (2p)}{n(1-{1}/{p}-g(n_0))}.
\end{align*}
This together with $n\geq n_0$ and $\al\geq {\displaystyle{\frac{1+\sqrt{5}}{2}}}$ gives 
\begin{align}\label{alRealU}
\log \left(\frac{1+\sqrt{5}}{2}\right) & \leq \log \al \leq \frac{\log (2p)}{n(1-{1}/{p}-g(n_0))}\nonumber\\
& \leq \begin{cases}
\frac{\log 2p}{n_0(1-{1}/{p}-g(n_0))}, &  \quad n_0<29;\\
\frac{\log (2\cdot 29)}{29(1-{1}/{29}-g(29))}, & \quad n_0=p\geq 29.
\end{cases} 
\end{align}
We check that the right--most quantity exceeds ${\displaystyle{\log\left( \frac{1+\sqrt{5}}{2}\right)}}$ except when 
$n_0=p\in \{5, 7\}$. Further, for $n_0=p\in \{5, 7\}$, putting $n=p\ell$,  we obtain by using \eqref{logal1},
\begin{align*}
\log (C_{m_k}/2)\leq n\log \al=p\ell\log \al\leq \frac{\log 2p}{1-{1}/{p}-g(p)}
\leq \begin{cases}
15.62, & {\rm if} \ p=5;\\
8.11, & {\rm if} \ p=7. 
\end{cases}
\end{align*}
This gives $m_k\leq 15, 9$ according to whether $n_0=p=5, 7,$ respectively. Further 
$1\leq \ell\leq 6, 2$, according as $p=5, 7$, respectively since $\al\geq {\displaystyle{\frac{1+\sqrt{5}}{2}}}$. 
This together with $P(U_n)\leq P(B_{m_k})$ yields  
$$\log \left(\frac{1+\sqrt{5}}{2}\right)\leq \log \al\leq \begin{cases}
\frac{15.62}{5\ell}, & {\rm if} \ p=5;\\
\frac{8.11}{7\ell}, & {\rm if} \ p=7
\end{cases} \quad {\rm and} \quad P(U_{p\ell})\leq 
\begin{cases}
29, & {\rm if} \ p=5;\\
19, & {\rm if} \ p=7. 
\end{cases} $$
For the pairs $(r, s)$ given by Lemma \ref{rsBD} with the conditions above, we check that the equation \eqref{eqn1}  
has no solution with $n=p\ell$.  Therefore, equation \eqref{eqn1} has no solution for $\al, \be$ real and 
$n\geq 5, n\notin \{6, 8, 12, 24\}$. 
  
\subsection{The case when $\al, \be$ are real and $n=24$}

Let $\al, \be$ be real and $n=24$. Then $|x|<1$. We have
\begin{align}\label{UD24}
\log \cD=\log U_{24}=\log \left(\frac{\al^{24}-\beta^{24}}{\al-\beta}\right)=
23\log \al+\log\left( \frac{1-x^{12}}{1-x}\right)+\log (1+x^{12}). 
\end{align}
We take $n_0=n=24$. Let $g>0$ and $\lambda\geq 0$ be such that 
\begin{align}\label{Dub24}
M_{24}(\cD) \leq g\left(\log \cD-\lambda) \leq g(23\log \al+\log\left|\frac{1-x^{12}}{1-x}\right|+\log |1+x^{12}|-\lambda\right).
\end{align}
In particular,
$$
g\leq g_0(24)=\frac{2.746}{8}\quad {\text{\rm  and}}\quad \lambda=0,
$$ 
by \eqref{Dub}. We now take $n_0=24, t=0$ in \eqref{n_0U} to get a  
lower bound for $M_{24}(U_{24})=M_{24}(\cD)$ and compare it with \eqref{Dub24} to obtain 
$$8\log |\al|+\log |1+x^{12}|-\log 6
\leq g\left(23\log \al+\log \left|\frac{1-x^{12}}{1-x}\right|+\log |1+x^{12}|-\lambda\right).$$
This gives 
$$(8-23g)\log \al\leq g\left(\log \left|\frac{1-x^{12}}{1-x}\right|+\left(1-\frac{1}{g}\right)\log |1+x^{12}|+\frac{\log 6}{g}-\lambda\right).$$
Recall that $|x|<1$. Assume that $x<0$. Then 
$$
\frac{1-x^{12}}{1-x}=1+x\left(\frac{1-x^{11}}{1-x}\right)<1\quad {\text{\rm  and}}\quad 1+x^{12}>1,
$$ 
which together 
with $g<1$ implies  the right hand side of the above inequality is strictly less than $\log 6$. 

Assume next that $x>0$.  Then 
 $$
 \frac{1-x^{12}}{1-x}=1+x+x^2+\cdots x^{11}<12,
 $$ 
 since $x<1$.  For any $x_0$ with $0<x_0<1$, we have 
\begin{align*}
\log\left |\frac{1-x^{12}}{1-x}\right|& +\left(1-\frac{1}{g}\right)\log |1+x^{12}|\\
\leq& \begin{cases}
\log 12+\left(1-\frac{1}{g}\right)\log (1+x_0), & x>x_0^{\frac{1}{12}},\\
\log \left|\frac{1-x_0}{1-x_0^{\frac{1}{12}}}\right|, & x\leq x_0^{\frac{1}{12}}. 
\end{cases}
\end{align*}
Putting 
$$y_0:=y_0(g, x_0)=\frac{\log 6}{g}-\lambda+
\max \left(\log 12+\left(1-\frac{1}{g}\right)\log (1+x_0), \log \left|\frac{1-x_0}{1-x_0^{\frac{1}{12}}}\right|\right),$$
we get 
\begin{align}\label{g|al|}
\log \al<\frac{y_0g}{8-23g} \qquad {\rm or} \qquad \al<\exp\left(\frac{y_0g}{8-23g}\right).
\end{align}

As stated before, we have 
$$
g\leq g_0(24)=\frac{2.746}{8}<0.3433\quad {\text{\rm and}}\quad \lambda=0,
$$ 
by \eqref{Dub}. Taking $g=0.3433, \lambda=0$ and $x_0=0.298$, we get $y_0\leq 7.21$ and hence $\log \al<23.78$ by \eqref{g|al|}. 
However, for $m_k\geq 420$, we have 
$$
\log \al\geq \frac{\log (C_{m_k}/2)}{24}\geq \frac{\log (C_{240}/2)}{24}>24.82,
$$ 
by \eqref{logal1}. Thus, $m_k<420$. 

For each $j<420$, let  
\begin{align*}
\ep_{1j}:=\frac{M_{24}(C_j)}{\log C_j} \qquad {\rm and} \qquad \ep_{2j}:=\frac{M_{24}(B_j)}{\log B_j}.
\end{align*}
Then $\ep_{1j}=\ep_{2j}=0$ for $j<12$ since $23$ is the least prime congruent to one of $\pm 1$ modulo $24$. 
We check that $\max(\ep_{1j}, \ep_{2j})\leq 0.3433$ for $j<420$. Write $U_{24}=\cD=\prod^k_{i=1}D_{m_i}$ as 
$$
\log \cD=\sum_jt_{1j}\log C_j+\sum_jt_{2j}\log D_j,
$$ 
where 
$$
t_{1j}:=\#\{i: D_{m_i}=C_{m_i}\} \qquad  {\rm and} \qquad t_{2j}:=\#\{i: D_{m_i}=B_{m_i}\}.$$
Let 
$\ep_0\geq \max_j\{\ep_{1j}, \ep_{2j}\}$ for $j$ such that $t_{1j}+t_{2j}>0$. Then 
\begin{align*}
M_{24}(U_{24})&=\sum_j(\ep_{1j}t_{1j}\log C_{j}+\ep_{2j}t_{2j}\log B_{j})\\
&=\ep_0\sum_j\left(\left(1-\left(1-\frac{\ep_{1j}}{\ep_0}\right)\right)t_{1j}\log C_{j}\right.\\
& +
\left.\left(1-\left(1-\frac{\ep_{2j}}{\ep_0}\right)\right)t_{2j}\log B_{j}\right)\\
&\leq \ep_0\left(\log \cD-\sum_jt_{1j}\lambda_{1j}-\sum_jt_{2j}\lambda_{2j} \right),
\end{align*}
where 
$$\lambda_{1j}:=\left(1-\frac{\ep_{1j}}{\ep_0}\right)\log C_j
 \quad {\rm and} \quad \lambda_{2j}:=\left(1-\frac{\ep_{2j}}{\ep_0}\right)\log B_{j}.$$
It is clear that $\lambda_{1j}\geq 0$ and $\lambda_{2j}\geq 0$.  

Suppose that $\al\leq 100$. Then $t_{1j}+t_{2j}>0$ implies $j\leq m_k\leq 85$ by \eqref{logal1} since 
$\log (C_{86}/2)>24\log 100$.  For $j\leq 85$ and $j\notin \{37, 38, 39, 40, 41, 42, 43\}$, we find that 
$\ep_{1j}, \ep_{2j}\leq 0.29$. Taking $g=0.29$ and $\lambda=0$ in 
\eqref{Dub24} and taking $x_0=0.25$, we get $y_0\leq 8.12$ and $\al<5.88$ so $\log \al\leq 1.771$. 
By \eqref{logal1} again, we have $j\leq m_k\leq 32$ since $\log (C_{33}/2)>24\log 5.88$ and we furthermore have
$P(U_{24})\leq P(B_{32})\leq 61$.  We check that the equation \eqref{eqn1} with $P(U_{24})\leq 61$ and 
$\al\leq 5.88$ is not possible. Here, we use Lemma \ref{rsBD}  to find all possible pairs $(r, s)$ with $\al\leq 5.88$. 
Thus, we assume  $t_{1j}+t_{2j}>0$ for some 
$j\in \{37, 38, 39, 40, 41, 42, 43\}$. Also 
$$\log \al\geq \frac{\log (C_{37}/2)}{24}>2.0379,
$$ 
by \eqref{logal1}. 
Again $t_{1j}+t_{2j}>0$ implies $j\leq m_k\leq 46$ since 
$$\log C_{37}+\log C_{47}>24\log 100+\log 2\geq \log U_{24}.
$$  
Hence, $P(U_{24})\leq P(B_{46})\leq 89$. Further $47\cdot 53\cdot 59\cdot 61\cdot 67\cdot 71\cdot 73\mid U_{24}$ 
also since $47\cdot 53\cdot 59\cdot 61\cdot 67\cdot 71\cdot 73\mid C_{j}\mid B_j$ for $j\in \{37, 38, 39, 40, 41, 42, 43\}$. 
For the pairs $(r, s)$ with $2.0379<\log \al\leq \log 100$ given by Lemma \ref{rsBD}, we check that 
$P(U_{24})\leq 89$ and $47\cdot 53\cdot 59\cdot 61\cdot 67\cdot 71\cdot 73\mid U_{24}$  is not possible. Therefore, 
 equation \eqref{eqn1} has no solution when $\al\leq 100$.

From now on, we assume that $\al>100$. Suppose that $t_{1j}=0$ for $j\in \{37, 38\}$. Then we find that 
$\max(\ep_{1j}, \ep_{2j})\leq \ep_0=0.324$ for $j<420$ with $j\neq 37, 38$, and also $\ep_{2j}<0.324$ for $j=37, 38$. 
By taking $g=0.324$ and $\lambda=0$ in \eqref{Dub24} and further $x_0=0.28$, we get $y_0\leq 7.5$ and $\al<84.3$. 
This is not possible. Therefore, we have $t_{1j}>0$ for $j=37$ or $j=38$. Then $\max(\ep_{1j}, \ep_{1j})\leq \ep_0=0.3433$ for 
$j<420$. Taking $g=\ep_0=0.3433$ and $\lambda=\sum_jt_{1j}\lambda_{1j}+\sum_jt_{2j}\lambda_{2j}$ in \eqref{Dub24} and 
further taking $x_0=0.298$, we obtain $y_0\leq 7.21-\lambda$ and 
\begin{align*}
\log \al<\frac{0.3433\left(7.21-\sum_jt_{1j}\lambda_{1j}-\sum_jt_{2j}\lambda_{2j}j\right)}{8-23\times 0.3433}.
\end{align*}
Together with $\al>100$, this gives 
\begin{align}\label{24main}
\sum_jt_{1j}\lambda_{1j}+\sum_jt_{2j}\lambda_{2j}\leq 7.21-\left(\frac{8}{0.3433}-23\right)\log 100\leq 5.8136.
\end{align}
We compute the values of $\lambda_{1j}$ and $\lambda_{2j}$ for $j<420$ and find that 
$$
\lambda_{1j}\leq 5.8136 \quad {\rm for} \quad j\in T_1:=\{j: j\leq 6\}\cup \{12, 13, 14, 37, 38, 39, 40, 41\},
$$
and 
$$
\lambda_{2j}\leq 5.8136 \quad {\rm for} \quad j\in  T_2:=\{j: j\leq 5\}\cup \{12, 37, 38\}.
$$
Thus, by \eqref{24main},  we may suppose that $t_{1j}>0$ implies $j\in T_1$ and $t_{2j}>0$ implies $j\in T_2$.  
Recall that we have $t_{1j}>0$ for  $j=37$ or $j=38$.  Write $t_4, t_5, t_{12}$ for $t_{1j}$ 
according to whether $j=4, 5, 12$, respectively. We find that $\lambda_{1j}\geq 2.639, 3.737, 3.111$ according to whether 
$j=4, 5, 12$, respectively. Hence, from \eqref{24main}, we have 
$t_4\leq 2, t_5\leq 1$ and $t_{12}\leq 1$. Put   
\begin{align*}
\log \cD_1:=\sum_{j\in T_1, j\neq 4, 5, 12}t_{1j}\log C_j+\sum_{j\in T_2}t_{2j}\log B_j,
\end{align*}
so that 
\begin{align}\label{80}
\log \cD&=\log \cD_1+t_4\log C_4+t_5\log C_5+t_{12}\log C_{12}.
\end{align}
We now consider $M_8(\cD)$ given by \eqref{Mn}. 
We find that $M_8(C_j)<0.46\log C_j$ for all $j\in T_1$ except when $j\in \{ 4, 5, 12\}$ and  
$M_8(B_j)<0.46\log B_j$ for $j\in T_2$ and further 
$$M_8(C_4)\leq 0.74, \quad M_8(C_5)\leq 0.53 \quad {\rm and} \quad M_8(C_{12})\leq 0.65. 
$$ 
Hence, from \eqref{80}, \eqref{UD24} and the fact that ${\displaystyle{\frac{1-x^{12}}{1-x}<12}}$, we get 
\begin{align*}
M_8(\cD)< &0.46\log \cD_1+0.74t_4\log C_4+0.53t_5\log C_5+0.65t_{12}\log C_{12}\\
<&0.46\log \cD+0.28t_4\log C_4+0.07t_5\log C_5+0.19t_{12}\log C_{12}\\
< &0.46(23\log \al+\log 12+\log (1+x^{12}))\\
&+0.56\log C_4+0.07\log C_5+0.19\log C_{12},
\end{align*}
since $t_4\leq 2, t_5\leq 1$ and $t_{12}\leq 1$. 
Comparing the above inequality with the lower bound of $M_8(\cD)=M_8(U_{24})$ given by \eqref{n_0U} with 
$n_0=2^3$ and $t=0$, 
we obtain 
\begin{align*}
4.07>&0.56\log C_4+0.07\log C_5+0.19\log C_{12}\\
>&(12-0.46\times 23)\log \al+(1-0.46)\log (1+x^{12})-0.46\log 12-\log 2\\
>&(12-0.46\times 23)\log 100-0.46\log 12-\log 2>4.7
\end{align*}
since $1+x^{12}>0$ and $\al>100$. This is a contradiction. Therefore, 
equation \eqref{eqn1} has no solution with $n=24$ when $\al$ and $\be$ are real. 

\subsection{The case of equation \eqref{eqn2}}
We now consider the equation \eqref{eqn2}.  Since $V_n=U_{2n}/U_n$, we see that primitive 
divisors of $V_n$ are the primitive divisors of $U_{2n}$. From the table listed in the beginning of 
Section 2, we find that  the values of $n\geq 4$ for which $V_n$ does not have a primitive divisor which are given by 
 the instances for which $U_{2n}$ has no primitive divisors  belongs to the set  $\{4, 5, 6, 9\}$. For 
 $n\in \{4, 5, 6, 9\}$ and corresponding pairs $(r, s)$ (which are given by pairs $(r, s)$ corresponding to 
 $2n$ in the table),  we check that  the equation \eqref{eqn2} has no solution.  Hence, for the proof of 
 Theorem 1,  we now assume that $n\geq 4$ and further $V_n$  has a primitive divisor which is  
 congruent to one of  $ \pm 1$ modulo $2n$. 
 
Let $n=4t$ be even. Then 
$$V_{4t} = \alpha^{4t}+\beta^{4t}=(\alpha^{2t}+\beta^{2t})^2-2(\alpha\beta)^{2t}=V^2_{2t}-2(-s)^{2t}.$$ 
For an odd prime $p\mid V_{4t}$, we see that $2$ is a quadratic residue modulo $p$ and hence 
$p\equiv \pm 1\pmod{8}$. We observe that both $C_m$ and $B_m$ are divisible by each prime 
$m+1<p\leq 2m$. By Lemma \ref{pi<30}, there is a prime $p\equiv \pm 5\pmod 8$ with 
$m+1<p\leq 2m$ for each $m\geq 6$.  Thus, equation \eqref{eqn2} implies $m_k\leq 5$ 
which together with the fact that $V_n$ has a primitive prime divisor gives $n=4t=4$.  Further, $\gcd(V_t, s)=1$ 
for all $t\geq 1$ gives $\nu_2(V^2_{2t}-2(s)^{2t})\leq 1$ implying $\nu_2(V_4)\leq 1$. Considering 
$\nu_2(B_m), \nu_2(C_m)$ for $2\leq m\leq 5$ and using the fact that $V_4$ has a primitive prime divisor which is congruent to $\pm 1\pmod 8$, 
we get $V_4=C_4=14$. Now 
 $14=V_4=\al^4+\be^4=r^2(r^2+4s)+2s^2$ and $\gcd(V_4, s)=1$ gives $s$ odd and 
 $r$ even. Reducing the above relation modulo $8$, we get $14\equiv 2\pmod {8}$ which is a contradiction. 
 Thus, equation \eqref{eqn2} does not have a solution for $n$ even with $4\mid n$. 

From now on, we take $n$ odd with $n\geq 5$ or $2||n$ with $n\geq 6$. We have 
$$\log \cD=\sum^k_{i=1}\log D_{m_i}=\log |V_n|=\log |\al^n+\be^n|=n\log |\al|+\log |1+x^n|.$$ 
Since $V_n$ has a primitive divisor which is  congruent to one of $\pm 1$ modulo $2n$, we have 
$2m_k-1\geq 2n-1$ or $m_k\geq n$. By Lemma \ref{m14} and since $|1+x^n|\leq 2$, we have 
\begin{align}\label{logalV}
\log |\al|+\log 2\geq \log |V_n|\geq \log 2+\log C_{m_k/2}\geq \log 2+1.38n \ {\rm for} \ n\geq 2100.
\end{align}
Let $p$ be an odd prime and $h>0, t\geq 0$ be such that $p^{h+t}\mid n$. 
Taking $n_0=p^h$, we use estimate \eqref{n_0V} of Lemma \ref{primn_0} to get a lower bound 
for the quantity $M_{n_0}(\cD)=M_{n_0}(V_n)$ and compare it with the upper bound given by Lemma \ref{MnBD}  
to obtain 
\begin{align*}
g(n_0)(n\log |\al|+\log |1+x^n|)\geq \left(1-\frac{1}{p^{t+1}}\right)n\log |\al|+\log \left|\frac{1+x^n}{1+x^{n/p}}\right| 
-\log p^{t+1}, 
\end{align*}
implying 
\begin{align}\label{DVcmp}
&\left(1-\frac{1}{p^{t+1}}-g(n_0)\right)
\leq & \frac{(g(n_0)-1)\log |1+x^n|+\log |1+x^{n/p^{t+1}}|+\log(p^{t+1})}{n\log |\al|},
\end{align}
where $g(n_0)$ is given by \eqref{gn}. We consider different cases as in  the analysis for $U_n$. 

Let $\al$ and $\be$ be complex conjugates. We may assume that $n\geq 6500$. We choose 
$n_0=p^h$ and $t$ given  by \eqref{nocomp}.  Assume that $\log |\al|\leq 4$.  
Then using 
$$|1+x^n|=\frac{|V_n|}{|\al|^n}, \quad  |1+x^{n/p^{t+1}}|\leq 2 \quad {\rm and} \quad \log |\al|\leq 4,$$
along with \eqref{logalV} in \eqref{DVcmp}, we obtain
 \begin{align*}
0\geq &\frac{(1-g(n_0))\log |V_n|+\log (2p^{t+1})}{n\log |\al|}-\frac{1}{p^{t+1}}\\
\geq &\frac{1.38(1-g(n_0))+\log (2p^{t+1})}{4n}-\frac{1}{p^{t+1}},
\end{align*}
which is the inequality \eqref{al<4U}. As in the case of $U_n$ in Section 4.2, we a get a contradiction. 
Assume now that $\log |\al|>4$. By Lemma \ref{-44.72}, we get 
$$
\log |1+x^n|\geq -f(n)\log |\al|,
$$ 
where 
$f(n)$ be given by \eqref{f(x)}. Using this along with $|1+x^{n/p^{t+1}}|\leq 2$ and $n\log |\al|> 4n$ (since $\log |\al|>4$) 
in \eqref{DVcmp}, we obtain the inequality \eqref{al>4U}.  As in the case of $U_n$ in Section 4.2, we a get a contradiction.  
Therefore, equation \eqref{eqn2} has no solution $n\geq 6500.$

Let $\al$ and $\be$ be real. Then $\alpha>0$. We take $n\geq 5, n\neq 6$. We choose $n_0=p^h$ and $t$ given  
by \eqref{noreal}, $n_0\neq 2^4$. Since $p$ is odd, writing 
$$
1+x^n=(1+x^{n/p})\left(\frac{1+x^n}{1+x^{n/p}}\right),
$$ 
we have  
 \begin{align*}
 |1+x^{n/p}|\leq 2 \quad {\rm and}  \quad \frac{1+x^n}{1+x^{n/p}}=\begin{cases}
 1+\sum^{p-1}_{i=1}(-y)^{i}>1, &  y=x^{\frac{n}{p}}<0;\\
 \frac{1+y^p}{1+y}\geq \frac{1}{1+y}>\frac{1}{2}, &  y=x^{\frac{n}{p}}>0.
\end{cases}
 \end{align*}
 Using this in \eqref{DVcmp}, we obtain 
 \begin{align*}
\log \al &\leq  \frac{(g(n_0)-1)\log\left( \frac{1+x^n}{1+x^{n/p}}\right)+g(n_0)\log (1+x^{n/p})+\log p}{
n(1-{1}/{p}-g(n_0)))}\\
 &\leq \frac{(1-g(n_0))\log 2+g(n_0)\log 2+\log p}{n(1-\frac{1}{p}-g(n_0))}
=  \frac{\log (2p)}{n(1-{1}/{p}-g(n_0))}.
\end{align*}
This together with $n\geq n_0$ and ${\displaystyle{\al\geq \frac{1+\sqrt{5}}{2}}}$ gives  \eqref{alRealU}. 
As in the case of $U_n$ in Section 4.3, we a get a contradiction except for $n_0=p\in \{5, 7\}$. 
Further, for $n_0=p\in \{5, 7\}$, putting $n=p\ell$, we obtain similarly   
$$\log \left(\frac{1+\sqrt{5}}{2}\right)\leq \log \al\leq \begin{cases}
\frac{15.62}{5\ell}, & {\rm if} \ p=5;\\
\frac{8.11}{7\ell}, & {\rm if} \ p=7,
\end{cases} \quad {\rm and} \quad P(U_{p\ell})\leq 
\begin{cases}
29, & {\rm if} \ p=5;\\
19, & {\rm if} \ p=7. 
\end{cases} $$
For the pairs $(r, s)$ given by Lemma \ref{rsBD} with the above conditions, we check that the equation \eqref{eqn2}  
has no solution at $n=p\ell$.  Therefore, equation \eqref{eqn2} has no solution for $\al, \be$ real and 
$n\geq 5,~ n\neq 6$.

Writing $n=p\ell$, we have from $P(V_n)\leq P(B_{m_k})$ that 
$$\log\left( \frac{1+\sqrt{5}}{2}\right)\leq \log \al\leq \begin{cases}
\frac{15.62}{5\ell}, & {\rm if} \ p=5;\\
\frac{8.11}{7\ell}, & {\rm if} \ p=7,
\end{cases} \quad {\rm and} \quad P(V_{p\ell})\leq 
\begin{cases}
29, & {\rm if} \ p=5;\\
19, & {\rm if} \ p=7. 
\end{cases} $$
For the pairs $(r, s)$ given by Lemma \ref{rsBD} with the conditions above, we check that the equation \eqref{eqn1}  
has no solutions.

\qed 

\section{The Proof of Theorem \ref{thm:3}}

First we prove the following result for $s=\pm 1$. 

\begin{lemma}\label{V23}
Let $s\in \{\pm 1\}$ and $r\geq 1$. Then $V_n\in \{C_m, B_m, 2C_m, 2B_m\}$ with $n>1$ and $m>1$ implies $n=3$ or 
\begin{align}\label{V20}
\begin{split}
n=2:\, &\, (r, s; V_2)=(2, 1; B_2);\\
n=2:\, &\, (r, s; V_2)=(4, -1; C_4);\\
n=3:\, &\, (r, s; V_3)=(1, 1; 2C_2), (2, 1; C_4), (5, 1; 2B_4).
\end{split}
\end{align}
\end{lemma}
\begin{proof}
Let $m\geq 2$ and 
$$
\mathcal{D}_m:=\{C_m, B_m, 2C_m, 2B_m\}.
$$
By Theorem \ref{theorem1} and $C_2=2$,  we have $V_n\in {\mathcal D}_m$ implies $n\in \{1, 2, 3, 6\}$.  Let 
$n\in \{2, 6\}$ and $V_n\in {\mathcal D}_m$. Now $V_2=r^2+2s$ and $V_6=(r^2+2s)((r^2+2s)^2-3)$. 
Let $p\equiv 5\pmod {12}$ be a prime such that $p\mid V_2V_6$.  Then we either have
$r^2\equiv -2s\pmod p$ or $(r^2+2s)^2\equiv 3\pmod p$. This is not possible since both 
${\displaystyle{\left(\frac{\pm 2}{p}\right)=\left(\frac{3}{p}\right)=-1}}$, where $(\frac{\cdot}{\cdot})$ is the Legendre symbol. 
Thus, $p\nmid V_2V_6$ for any prime $p\equiv 5\pmod {12}$. By Lemma \ref{pi<30}, we get $m\leq 8$. 
Further from $s=\pm 1$, we have $\nu_2(r^2+2s)\leq 1$ giving $\nu_2(V_2)=\nu_2(V_6)\leq 1$. 
Using both $5\nmid V_2V_6$ and $\nu_2(V_2)=\nu_2(V_6)\leq 1$, we find that if 
$V_2, V_6\in {\mathcal D}_m$ with $2\leq m\leq 14$, then $V_2, V_6\in \{C_m\}$ implies 
$m\in \{2, 4, 5, 7\};$  $V_2, V_6\in \{2C_m\}$ implies $m=7$; $V_2, V_6\in \{B_m\}$ implies $m=2$ 
and $V_2, V_6\notin \{2B_m\}$. Now $V_2=r^2+2s=E_m\in {\mathcal D}_m$ gives $r^2=E_m-2s$. We check that 
for the values $m\in \{2, 4, 5, 7\}$,  $C_m-2s$ is a square only when $r=2, s=-1, C_2=2$ and 
$r=4, s=-1, C_2=2$; $B_2-2s=6-2s$ is a square only for $r=2, s=1$; and $2C_7-2s$ is not a square. 
These solutions are listed in \eqref{V20} except that we omit $(r,s)=(2,-1)$ since it gives a degenerate characteristic equation.  Let $V_6=(r^2+2s)((r^2+2s)^2-3)=E_m\in {\mathcal D}_m$. 
We check that for $r\leq 3$, $V_6=E_m\in {\mathcal D}_m$ only for $r=2, s=-1, V_6=C_2=2$ and 
we omit $(r,s)=(2,-1)$. For $r+2s\geq 4$, we have $r_1=r^2+2s\geq 14$ and hence $(r_1-2)^3<r_1(r^2_1-3)=V_6\leq C_{7}=429$. 
This gives $r_1=r^2+2s\leq 9$, which is not possible. 

Let $n=3$ and $V_3=r(r^2+3s)=E_m\in {\mathcal D}_m$. For $r\leq 3$, we check that indeed $V_3=E_m$ only for 
the pair $(r, s)=(2, 1)$. We now take $r\geq 4$. Using the inequality $r^3<r(r^2+3)=E_m<(r+1)^3$ when $s=1$
and $(r-2)^3<r(r^2-3)=E_m<(r-1)^3$ when $s=-1$, we get $r=\lfloor E^{1/3}_m\rfloor$ when $s=1$ and 
$r=\lfloor E^{1/3}_m+2\rfloor$ when $s=-1$. For $m\leq 15$, we find that putting $r=\lfloor E^{1/3}_m\rfloor$ gives 
$r(r^2+3)=E_m$ only when $r=2, ~E_m=C_4=7$ and $r=5,~ E_m=2B_4=140$ which are already listed in \eqref{V20} (except that we again omit $(r,s)=(2,-1)$) and 
putting $r=\lfloor E^{1/3}_m+2\rfloor$ gives $r(r^2-3)\neq E_m$. Thus, we assume that $m\geq 16$.  

We observe that $\nu_2(r^2+3)\leq 2, \nu_2(r^2-3)\leq 1$ and $\nu_3(r^2\pm 3)\leq 1$. Further we observe 
that primes $p\mid r^2+3s$ with $p>3$ satisfy ${\displaystyle{\left(\frac{-3s}{p}\right)=1}}$. We get $p\equiv 1, 7\pmod {12}$ if $p>3$ when 
$s=1$ and $p\equiv \pm 1\pmod {12}$ if $p>3$ when $s=-1$.  We take $n_0=12$ and $\ell_0=7, -1$, 
according to whether $s=1, -1$, respectively. From the equation 
$$
V_3=r(r^2+3s)=E_m, \quad E_m\in {\mathcal D}_m,
$$ 
we obtain 
$$\log (r^2+3s)=2\log r+\log\left(1+\frac{3s}{r^2}\right)\leq \xi(m)+\log 3+\begin{cases}
\log 4, & {\rm if} \quad s=1;\\
\log 2, & {\rm if} \quad s=-1,
\end{cases}$$
where 
\begin{align*}
\xi(m)=\displaystyle{\sum _{ p\equiv 1, \ell_0\pmod {12}}}\nu_p(2B_m)\log p=
\displaystyle{\sum _{ p\equiv 1, \ell_0\pmod {12}}}\nu_p(B_m)\log p.
\end{align*}
From 
$$\log E_m=\log r(r^2+3s)=\frac{3}{2}\left(2\log r+\log\left(1+\frac{3s}{r^2}\right)\right)-\frac{\log(1+{3s}/{r^2})}{2},$$
and 
$$\frac{3}{2}\log 4-\log (1+3s/r^2)<\begin{cases}
\frac{3}{2}\log 4, & {\rm if} \quad s=1;\\
\frac{3}{2}\log 2-\log(1-3/16) <\frac{3}{2}\log 4, & {\rm if} \quad s=-1,
\end{cases}$$
since $r\geq 4$, together with $E_m\geq C_m$, we get 
$$\log C_m\leq \log E_m<\frac{3}{2}\left(\xi(m)+\log 12\right)  \quad {\rm implying} \quad 
\frac{2}{3}<\frac{\xi(m)}{\log C_m}+\frac{\log 12}{\log C_m}.$$
Hence, 
\begin{align}\label{2/3}
\log C_m<\frac{\log 12}{\frac{2}{3}-\frac{\xi(m)}{\log C_m}}.
\end{align}
For $16\leq m\leq  35$, we find that $\frac{\xi(m)}{\log C_m}<0.52$ and therefore 
$$
\log C_{m}<\frac{\log 12}{\frac{2}{3}-0.52}<\log C_{16},
$$ 
which is a contradiction. Thus, we have 
$m\geq 36$. For $36\leq m<1500$, we check that $\frac{\xi(m)}{\log C_m}<0.59$ and hence 
$\log C_{m}<\frac{\log 12}{\frac{2}{3}-0.59}<\log C_{36}$, which is a contradiction again. Thus, 
$m\geq 1500$. As in the proof of Lemma \ref{MnBD}, we get 
\begin{align*}
\begin{split}
\xi(m)&\leq \sum _{\substack{(2m)^{1/2}<p\leq 2m \\ p\equiv 1, \ell_0\pmod {12}}}
\left( \left\lfloor \frac{2m}{p}\right\rfloor-2\left\lfloor \frac{m}{p}\right\rfloor\right)\log p \\
& + \sum _{\substack{p\leq (2m)^{1/2}\\ p\equiv 1, \ell_0\pmod {12}}}\left\lfloor \frac{\log (2m)}{\log p}\right\rfloor \log p\\
& \leq \sum_{\ell\in \{1, \ell_0\}}\left\{\psi(2m; 12, \ell)+\sum^3_{t=2}\theta(2m/t; 12, \ell)-\sum^2_{t=1}\theta(m/t; 12, \ell)\right\}. 
\end{split}
\end{align*}
The above inequality with Lemma \ref{pi<30}, yields
$$
\xi(m)\leq \frac{\delta_{1}m}{4}=\frac{\delta_1m}{4\log C_m}\log C_m
\quad {\rm   for} \quad  m\geq 1500,
$$ 
where 
\begin{align*}
\delta_{1}&=\frac{47}{15}+\frac{2\sqrt{2}\times 0.863}{\sqrt{1500}}\left(1+\frac{1}{\sqrt{3}}+\frac{1}{\sqrt{5}}\right)
+\frac{2\times 1.5}{\sqrt{1500}}\left(1+\frac{1}{\sqrt{2}}\right).
\end{align*}
Hence, 
$$\frac{\xi(m)}{\log C_m}\leq \frac{\delta_1}{4}\frac{m}{\log C_m}
\leq \frac{\delta_1}{4}\frac{1500}{\log C_{1500}}<0.62,
$$
by Lemma \ref{m14}. Inserting this last estimate into \eqref{2/3}, we get 
$$\log C_m<\frac{\log 12}{\frac{2}{3}-0.62}<54<\log C_{50}.$$
This is a contradiction and the proof of Lemma \ref{V23} is complete.
\end{proof}

\noindent 
{\bf Proof of Theorem \ref{thm:3}:} Let $d$ be a squarefree positive integer and assume  $\varepsilon\in \{\pm 1\}$. 
Let $(X_n, Y_n)$  be the $n$th solution of the equation $X^2-dY^2=\varepsilon$. Then 
$X_n=(\alpha^n+\beta^n)/2$ where $(\alpha,\beta)$ are the two roots of the quadratic 
$x^2-(2X_1)x+\varepsilon=0$. Observe that $C_2=2$. Thus, $X_n\in \{C_m, B_m\}$ gives 
$V_n\in \{2C_m, 2B_m\}$ where $V_n=\alpha^n+\beta^n$ and $s=-\alpha\beta=\pm 1$. Then for $n>1$, $V_n$ is given by 
Lemma \ref{V23}, namely $n=3$,  $V_n\in \{2C_2, 2B_4\}$. Then   
$X_n=V_n/2\in \{C_2=2, B_4=70\}$.  The solutions given by 
\begin{align*}
&2^2-3\cdot 1^2=1, \qquad \hspace{1.5cm} \quad 2^2-5\cdot 1^2=-1\\
 {\rm and} \qquad & 70^2-3\cdot  23 \cdot 71\cdot 1^2=1, \qquad \quad  70^2-29\cdot 13^2=-1, 
\end{align*}
are exactly $(X_1, Y_1)$ of the corresponding Pell equations and the assertion of Theorem \ref{thm:3} for 
$X_n$ follows. 

We now consider solutions $(W_n, Z_n)$ of $W^2-dZ^2=4\varepsilon$ with $\varepsilon\in \{\pm 1\}$. 
Assume that $W_n\in \{C_m, B_m\}$. Note that by putting 
$$
\alpha:=(W_1+{\sqrt{d}}Z_1)/2\quad {\text{\rm and}}\quad \beta:=(W_1-{\sqrt{d}}Z_1)/2,
$$
we have that $W_n=\alpha^n+\beta^n=V_n$ and $s=-\alpha\beta=\pm 1$. By Lemma \ref{V23}, we get that 
either $n=1$ or $n=2$ with $V_2\in \{C_2=2, B_2=6, C_4=14\}$ or 
$n=3$ with $V_3\in \{C_2=2, C_4=14\}$ or $n=6$ with $V_6=C_2=2$. For $n\neq 1$, we have solutions
\begin{align*}
&d=2, (W_2,Z_2)=(B_2, 4) \ {\rm with} \  6^2-2\cdot 4^2=4 \ ({\rm and} \ (W_1,Z_1)=(2, 2)); \\
&d=2, (W_2, Z_2)=(C_4, 10)\ {\rm with} \  14^2-2\cdot 10^2=-4 \ ({\rm and} \ (W_1,Z_1)=(2, 2));\\
&d=3, (W_2, Z_2)=(C_4, 8)\ {\rm with} \  14^2-3\cdot 8^2=4 \ ({\rm and} \ (W_1,Z_1)=(4, 2)). 
\end{align*}
The solution given by $B^2_2-10\cdot 2^2=6^2-40=-4$ is exactly $(W_1, Z_1)=(6, 2)$ for $d=10$. 
 This finishes the proof of Theorem \ref{thm:3}. 
\qed
\section*{Acknowledgements}

Part of this work was done when the first author visited the School of Maths of Wits University in December 2018. 
 He thanks this Institution for hospitality and CoEMaSS Grant RTNUM18 for financial support.  Part of this work was done when the second author visited the Max Planck Institute for 
 Mathematics in Bonn, Germany in Fall of 2019. This author thanks MPIM for hospitality.

\end{document}